\documentclass[10pt, a4paper, reqno]{amsart}
\usepackage[foot]{amsaddr}
\usepackage{lmodern}
\usepackage{graphicx} 
\usepackage{amsthm, amsmath, amsfonts, amssymb, amscd, enumerate, enumitem, esint, xcolor, url, mathrsfs}
\usepackage[colorlinks]{hyperref}
\usepackage[utf8]{inputenc}
\usepackage{mathtools}
\usepackage{enumerate}
\usepackage[noadjust]{cite}

\newtheorem{ThA}{Theorem}

\newtheorem{thm}{Theorem}[section]
\newtheorem{cor}[thm]{Corollary}
\newtheorem{lem}[thm]{Lemma}
\newtheorem{prop}[thm]{Proposition}
\theoremstyle{definition}

\theoremstyle{remark}
\newtheorem{rem}[thm]{Remark}

\DeclareMathOperator{\supp}{supp}

\newcommand{\RH}{\textup{RH}}

\numberwithin{equation}{section}
\usepackage[defaultlines=2, all]{nowidow}
\setnoclub[2]
\allowdisplaybreaks

\setlist[enumerate,1]{label=(\alph*)}

\begin{document}

	\title[On $m$-order logarithmic Schr\"odinger operator]
	{On $m$-order logarithmic Schr\"odinger operator}

	\author[J. J. Betancor]{Jorge J. Betancor$^1$}
        \address{$^1$Departamento de An\'alisis Matem\'atico, Universidad de La Laguna,\newline
	Campus de Anchieta, Avda. Astrof\'isico S\'anchez, s/n,\newline
	38721 La Laguna (Sta. Cruz de Tenerife), Spain}
	\email[J. J. Betancor]{jbetanco@ull.es}
	
	\author[E. Dalmasso]{Estefan\'ia Dalmasso$^{2}$}
	\address{$^2$Instituto de Matem\'atica Aplicada del Litoral, UNL, CONICET,     FIQ.\newline 
    Colectora Ruta Nac. Nº 168, Paraje El Pozo,\newline S3007ABA, Santa Fe, Argentina}
	\email[E. Dalmasso]{edalmasso@santafe-conicet.gov.ar}
	
	\author[P. Quijano]{Pablo Quijano$^{2}$}
	\email[P. Quijano]{pquijano@santafe-conicet.gov.ar}

     \author[L. Rodríguez-Mesa]{Lourdes Rodríguez-Mesa$^1$}  
        \email[L. Rodríguez-Mesa]{lrguez@ull.edu.es}
	
	\date{\today}
	

\begin{abstract}
In this paper we study the logarithm of order $m$ of the Schr\"odinger operator $\mathcal L_V$ in $\mathbb R^d$, for certain nonnegative potentials $V$. First, the operator $\log^m\mathcal L_V$, $m\in \mathbb N$, is defined by using the spectral measure associated with the self-adjoint operator $\mathcal L_V$ on a suitable subspace of $L^2(\mathbb R^d)$. Then, the semigroup of operators $\{T_t^V\}_{t>0}$ generated by $\mathcal L_V$ allows us to extend the definition of $\log^m\mathcal L_V$ to a wider class of Lipschitz functions. By using logarithmic operators $\log^m\mathcal L_V$, $m\in \mathbb N$, we prove Taylor expansions for the fractional powers $\mathcal L_V^s$ and $\mathcal L_V^{-s}$ with respect to the order $s\in (0,1)$, where the convergence is understood in $L^p(\mathbb R^d)$, $1<p<\infty$.   
\end{abstract}

\maketitle

\section{Introduction}

The logarithmic Laplacian operator, ${\rm log}(-\Delta)$, is a nonlocal operator that has been studied in various settings and has become an active area of research in recent years (\cite{ChChH, ChV2, ChV1, ChW, RCh, RChX, Feu1, Feu2, JSW, LW, Lee, Lee2, LL}). 
In \cite{ChW}, Chen and Weth defined the logarithmic Laplacian operator 
\[(\log(-\Delta)f)(x) = \frac{d}{ds}(-\Delta)^sf(x)\big\vert_{s=0},\] 
for every $f\in C_c^\infty(\mathbb R^d)$, the space of $C^\infty$ functions with compact support in $\mathbb R^d$, and established the following pointwise representation of $\log(-\Delta)$. 

\begin{ThA}[{\cite[Theorem~1.1]{ChW}}]\label{ThA}
Let $f\in C_c^\infty(\mathbb R^d)$. Then,
\begin{align*}
(\log(-\Delta)f)(x)
=c_d\int_{\mathbb R^d}\frac{f(x)\mathcal X_{B(x,1)}(y)-f(y)}{|x-y|^d}dy+\alpha_df(x),\quad x\in \mathbb R^d,
\end{align*}
where
$c_d=\pi ^{-d/2}\Gamma (d/2)$ and $\alpha _d=2\log 2+ \psi (d/2)-\gamma$, where $\gamma$ is the Euler--Mascheroni constant and $\psi =\Gamma'/\Gamma$ is the Digamma function.
\end{ThA}

Moreover, in \cite{Ch}, Chen studied the $m$-order logarithm of the Laplacian denoted by $\mathcal L_m$, for every $m\in \mathbb N$, which is the operator with the symbol $(2\log|x|)^m$ with respect to the Fourier transform. Motivated by the following well-known asymptotic expansion for $r>0$ and $m\in \mathbb N$,
\[
    r^s=1+\sum_{n=1}^m\frac{s^n}{n!}(\log r)^n+o(s^m),\quad \mbox{ as }s\rightarrow 0^+,
\]
and using the logarithmic Laplacian of higher order, Chen 
established asymptotic expansions for the fractional powers of the Laplacian $(-\Delta)^s$,  (see~\cite[Theorem~1.2]{Ch}), and for the Riesz potentials $(-\Delta)^{-s}$,  (see~\cite[Theorem~1.1]{Ch}), when $s\rightarrow 0^+$.

In this paper we study the $m$-order logarithm of the Schr\"odinger operator $\mathcal L_V$ in $\mathbb R^d$, where $V$ is a nonnegative potential (the first order logarithm of $\mathcal L_V$ has been analyzed in \cite{BDFQ} and \cite{BdLR}). Our objective is to obtain asymptotic expansions for the positive and negative fractional powers of the Schr\"odinger operator $\mathcal L_V$ through its higher order logarithmic operators.

We consider a nonnegative measurable function $V\in L^1_\textup{loc}(\mathbb R^d)$ and define the following closed and positive sesquilinear form $T$ by
\[
T(f,g)=\int_{\mathbb R ^d}\nabla f(x)\overline{\nabla g(x)}dx+\int_{\mathbb R^d}V(x)f(x)\overline{g(x)}dx,\quad f,g\in D(T),
\]
where $D(T)=\{f\in L^2(\mathbb R^d): \nabla f\in L^2(\mathbb R^d) \mbox{ and }\sqrt{V}f\in L^2(\mathbb R^d)\}$. Note that $D(T)$ is dense in $L^2(\mathbb R^d)$. According to  \cite[Theorem VIII.15]{RS}, there exists a self-adjoint operator $\mathcal L_V: D(\mathcal L_V)=D(T)\rightarrow L^2(\mathbb R^d)$ such that
\[
\langle \mathcal L_Vf,g\rangle=T(f,g),\quad f,g\in D(\mathcal L_V).
\]
Here $\langle \cdot,\cdot\rangle$ denotes the usual inner product in $L^2(\mathbb R^d)$. Note that $C_c^\infty (\mathbb R^d)\subset D(\mathcal L_V)$, and $\mathcal L_V\phi =(-\Delta +V)\phi$, for every $\phi \in C_c^\infty (\mathbb R^d)$.

Since $V\geq 0$, the spectrum $\sigma(\mathcal L_V)$ of $\mathcal L_V$ is contained in $[0,\infty )$. Moreover, if $V(x)>0$, a.e. $x\in \mathbb R^d$, $0$ is not an eigenvalue of $\mathcal L_V$. We assume in the sequel that $V(x)>0$, a.e. $x\in \mathbb R^d$. Then, there exists a unique spectral measure $E_V$ whose support is contained in $\sigma (\mathcal L_V)$ such that
\[
\mathcal L_Vf=\int_0^\infty \lambda dE_V(\lambda )f,\quad f\in D(\mathcal L_V),
\]
and $D(\mathcal L_V)=\{f\in L^2(\mathbb R^d): \int_0^\infty \lambda ^2d\mu_{f,f}^V(\lambda )<\infty\}$. Here, for every $f,g\in L^2(\mathbb R^d)$, $\mu _{f,g}^V$ denotes the Borel measure defined by $\mu_{f,g}^V(U)=\langle E_V(U)f,g\rangle$, for each Borel set $U\subset (0,\infty)$.

If $\Phi $ is a Borel measurable function in $\mathbb R$, we define the operator $\Phi(\mathcal L_V)$ by
\[
\Phi (\mathcal L_V)f=\int_0^\infty \Phi (\lambda )dE_V(\lambda )f,\quad f\in D(\Phi (\mathcal L_V)),
\]
where $D(\Phi (\mathcal L_V))=\{f\in L^2(\mathbb R^d): \int_0^\infty |\Phi (\lambda )|^2d\mu_{f,f}^V(\lambda )<\infty\}$. In particular, we define, for every $s\in \mathbb{R}\setminus\{0\}$, the $s$-power $\mathcal{L}_V^s$ of $\mathcal{L}_V$ by
\[
\mathcal{L}_V^sf=\int^\infty_0\lambda^sdE_V(\lambda)f,\quad f\in D(\mathcal{L}_V^s),
\]
being $D(\mathcal{L}_V^s)=\{f\in L^2(\mathbb{R}^d): \int_0^\infty\lambda^{2s}d\mu_{f,f}^V(\lambda)<\infty\}$, and the logarithmic Schr\"odinger operator $\log \mathcal L_V$ by
\[
(\log \mathcal L_V)f=\int_0^\infty \log \lambda dE_V(\lambda )f,\quad f\in D(\log \mathcal L_V), 
\]
where $D(\log  \mathcal L_V)=\{f\in L^2(\mathbb R^d): \int_0^\infty |\log \lambda |^2d\mu_{f,f}^V(\lambda )<\infty\}$.

For every $t>0$, we define the operator $T_t^V$ by
\[
T_t^Vf=\int_0^\infty e^{-\lambda t}dE_V(\lambda )f,\quad f\in L^2(\mathbb R^d).
\]
The family $\{T_t^V\}_{t>0}$ is a semigroup of contractions in $L^2(\mathbb R^d)$ whose infinitesimal generator is $\mathcal L_V$. For every $t>0$, there exists a measurable function $T_t^V:\mathbb R^d\times \mathbb R^d\rightarrow (0,\infty)$ such that
\begin{equation}\label{1.4}
T_t^V(f)(x)=\int_{\mathbb R^d}T_t^V(x,y)f(y)dy,\quad f\in L^2(\mathbb R^d). 
\end{equation}
The integral representation \eqref{1.4} allows us to extend, for every $t>0$, the operator $T_t^V$ from $L^2(\mathbb R^d)\cap L^p(\mathbb R^d)$ to $L^p(\mathbb R^d)$ as a contraction from $L^p(\mathbb R^d)$ into itself, for each $1\leq p<\infty$. Thus, $\{T_t^V\}_{t>0}$ is a semigroup of contractions in $L^p(\mathbb R^d)$, $1\leq p<\infty$.

We denote by $T_t(z)$, $t>0$ and $z\in \mathbb R^d$, the Euclidean heat kernel, that is, $T_t(z)=(4\pi t)^{-d/2}e^{-|z|^2/(4t)}$. By the Feynman-Kac formula we have
\begin{equation}\label{1.6}
0\leq T_t^V(x,y)\leq T_t(x-y),\quad x,y\in \mathbb R^d \mbox{ and }t>0.
\end{equation}

In \cite[Theorem 1.2]{BDFQ},  a pointwise representation of the logarithmic operator $\log \mathcal L_V$ for certain potentials $V$ was established. We say that $V\in L^1_\textup{loc} (\mathbb R^d)$ satisfies the $q$-reverse H\"older inequality, with $1<q<\infty$, denoted by $V\in \RH_q$, when there exists $C>0$ such that
\[
\left(\frac{1}{|B|}\int_BV(x)^qdx\right)^{1/q}\leq \frac{C}{|B|}\int_B V(x)dx,
\]
for every ball $B$ in $\mathbb R^d$. Functions in $\RH_q$ appeared in the studies of Gehring \cite{G} and Coifman and Fefferman \cite{CF}. It is common for the condition $V\in \RH_q$ to arise when harmonic analysis is developed in the Schr\"odinger setting. This can be seen, for instance, in \cite{BFHR, BHQ, Dz1, Dz2, DGMTZ, GHY, Sh1, ZT}.

Suppose in the sequel that $d\geq 3$ and $V\in \RH_q$ with $q>d/2$. We consider (see \cite{Sh2} and \cite[Definition 1.3]{Sh1}) the function $\rho :\mathbb R^d\rightarrow (0,\infty)$ given by
\[
\rho (x)=\sup\left\{r>0: \frac{1}{r^{d-2}}\int_{B(x,r)}V(y)dy\leq 1\right\}.
\]
This function $\rho$, also called the critical radius function, plays an important role in the study of harmonic analysis associated with Schr\"odinger operators. It holds that $\rho (x)>0$ for all $x\in \mathbb R^d$ (see \cite[\S1]{Sh1}).

When $V\in \textup{RH}_q$, with $q>d/2$, the estimate (\ref{1.6}) can be improved as follows (see \cite[Proposition 2]{DGMTZ}). For every $N\in \mathbb N$, there exists $C_N>0$ such that
\begin{equation}\label{acotTVimpr}
0\leq T_t^V(x,y)\leq \frac{C_N}{t^{d/2}}e^{-\frac{|x-y|^2}{5t}}\left(1+\frac{\sqrt{t}}{\rho (x)}+\frac{\sqrt{t}}{\rho (y)}\right)^{-N},\quad x,y\in \mathbb R^d\mbox{ and }t>0.
\end{equation}

A pointwise representation for the logarithmic Schr\"odinger operator $\log \mathcal L_V$ was established in \cite[Theorem 1.2]{BDFQ}.
\begin{ThA}\label{ThB}
Let $d\geq 3$ and $q>d/2$. Suppose that $V\in \RH_q$ and $f\in C_c^\infty (\mathbb R^d)$. We have that
\begin{align*}
    (\log \mathcal L_V)(f)(x)&=-\int_{B(x,1)}(f(y)-f(x))\int_0^\infty \frac{T_t^V(x,y)}{t}dtdy\\
    &\quad -\int_{\mathbb R^d\setminus B(x,1)}f(y)\int_0^\infty \frac{T_t^V(x,y)}{t}dtdy-K(x)f(x),
\end{align*}
for almost all $x\in \mathbb R^d$, where
\begin{align*}
  K(x)&=2\log \rho (x)+\int_0^{\rho (x)^2}\int_{\mathbb R^d}\frac{T_t^V(x,y)-T_t(x-y)}{t}dydt\\
  &-\int_0^{\rho (x)^2}\int_{\mathbb R^d\setminus B(x,1)}\frac{T_t^V(x,y)}{t}dydt+\int_{\rho (x)^2}^\infty \int_{B(x,1)}\frac{T_t^V(x,y)}{t}dydt+\gamma,\ x\in \mathbb R^d,
\end{align*}
where $\gamma$ denotes the Euler--Mascheroni constant.
\end{ThA}

Note that in contrast with the result in Theorem \ref{ThA}, the correction factor in Theorem \ref{ThB} is not constant. This difference is due to the fact that the semigroup $\{T_t^V\}_{t>0}$ is not Markovian, that is, $\int_{\mathbb R^d}T_t^V(x,y)dy$, $x\in \mathbb R^d$, does not coincide with the constant function $g(x)=1$, $x\in \mathbb R^d$, when $t>0$.

We now consider the operator $\mathbb{L}{\rm og}\, \mathcal L_V$ given by
\[
(\mathbb{L}{\rm og}\, \mathcal L_V)f=\frac{d}{ds}(\mathcal L_V^sf)\big\vert_{s=0^+}=\lim_{s\rightarrow 0^+}\frac{\mathcal L_V^sf-f}{s},\quad f\in D(\mathbb{L}{\rm og}\, \mathcal L_V),
\]
where the limit is understood in   $L^2(\mathbb R^d)$, 
being 
\begin{align*}
D(\mathbb{L}{\rm og}\, \mathcal L_V)&=\Bigg\{f\in L^2(\mathbb R^d): f\in D(\mathcal L_V^\beta),\mbox{ for some }\beta \in (0,1),\\
&\hspace{1cm}\mbox{ and }\lim_{s\rightarrow 0^+}\frac{\mathcal L_V^sf-f}{s}\mbox{ exists in }L^2(\mathbb R^d)\Bigg\}.
\end{align*}
In \cite[Theorem 1.1(a)]{BDFQ} it was proved that $(\log \mathcal L_V)f=(\mathbb{L}{\rm og}\, \mathcal L_V)f$, provided that $f\in D(\mathcal L_V^{s_0})\cap D(\log \mathcal L_V)$, for some $s_0\in (0,1)$. In a similar way, we can see that
\[
\lim_{s\rightarrow 0^+}\frac{\mathcal L_V^{-s}f-f}{s}=-(\log \mathcal L_V)f,\quad \mbox{ in }L^2(\mathbb R^d),
\]
for every $f\in D(\mathcal L_V^{-s_0})\cap D(\log \mathcal L_V)$, for some $s_0>0$. Moreover, in \cite[Theorem 1.2]{BDFQ} the following result was established .

\begin{ThA}\label{ThC}
    Let $d\geq 3$ and $q>d/2$. Suppose that $V\in \RH_q$ and $f\in C_c^\infty (\mathbb R^d)$. Then, for every $1<p\leq \infty$,
    \[
    \lim_{s\rightarrow 0^+}\frac{\mathcal L_V^sf-f}{s}=(\log \mathcal L_V)f,\quad \mbox{ in }L^p(\mathbb R^d).
    \]
\end{ThA}

We can write the above properties as follows:
\begin{equation}\label{1.2}
  \mathcal L_V^sf-f-s(\log \mathcal L_V)f=o(s),\quad \mbox{ as }s\rightarrow 0^+,\mbox{ in }L^2(\mathbb R^d),  
\end{equation}
when $f\in D(\mathcal L_V^{s_0})\cap D(\log \mathcal L_V)$, for some $s_0 \in (0,1)$, and 
\begin{equation}\label{1.3}
    \mathcal L_V^{-s}f-f+s(\log \mathcal L_V)f=o(s),\quad \mbox{ as }s\rightarrow 0^+,\mbox{ in }L^2(\mathbb R^d),   
\end{equation}
when $f\in D(\mathcal L_V^{-s_0})\cap D(\log \mathcal L_V)$, for some $s_0 \in (0,1)$.

We now aim to extend the asymptotic expansions in \eqref{1.2} and \eqref{1.3} to arbitrary order $m\in\mathbb{N}$, by means of higher order logarithmic operators.

We define, for every $m\in \mathbb N$, the operator $\log^m\mathcal L_V$ by
\[
(\log^m\mathcal L_V)f= \int_0^\infty (\log \lambda )^mdE_V(\lambda )f,\quad f\in D(\log^m\mathcal L_V),
\]
being $D(\log^m\mathcal L_V)=\{f\in L^2(\mathbb R^d): \int_0^\infty (\log \lambda )^{2m}d\mu _{f,f}^V(\lambda )<\infty \}$.

\begin{thm}\label{Th1.1}
Let $m\in \mathbb N$, $s_0\in (0,1]$, and $V$ be a measurable function such that $V(x)>0$ a.e. $x\in\mathbb R^d$.
\begin{enumerate}[label=\textnormal{(\alph*)}]
    \item \label{itm: Th1.1-a} If $f\in D(\log^{m+1}\mathcal L_V)\cap D(\mathcal L_V^{s_0})$, then 
\[
\mathcal L_V^sf-\sum_{k=0}^m\frac{s^k}{k!}(\log^k\mathcal L_V)f=o(s^m),\quad \mbox{ as }s\rightarrow 0^+,\;\mbox{ in }L^2(\mathbb R^d).
\]
\item \label{itm: Th1.1-b} If $f\in D(\log ^{m+1} \mathcal L_V)\cap D(\mathcal L_V^{-s_0})$, then
\[
\mathcal L_V^{-s}f-\sum_{k=0}^m(-1)^k\frac{s^k}{k!}(\log^k\mathcal L_V)f=o(s^m),\quad \mbox{ as }s\rightarrow 0^+,\; \mbox{ in }L^2(\mathbb R^d).
\]
\end{enumerate}
\end{thm}

Let $s\in (0,1)$. For every $f\in L^2 (\mathbb R^d)$ and $m\in \mathbb N$, the function $F_f: (1/m,\infty)\rightarrow L^2(\mathbb R^d)$ defined by
\[
F_f(t)=\frac{T_t^Vf-f}{t^{s+1}},
\]
is weakly measurable in $(1/m,\infty)$ and, since $L^2(\mathbb R^d)$ is a separable Banach space, according to \cite[Ch.~V, \S4, Theorem,  p.~131]{Yo}, $F_f$ is strongly measurable in $(1/m,\infty)$. On the other hand, 
\[
\int_{1/m}^\infty\frac{\|T_t^Vf-f\|_{L^2(\mathbb R^d)}}{t^{s+1}}dt<\infty,
\]
for every $f\in L^2(\mathbb R^d)$ and $m\in \mathbb N$, because $\{T_t^V\}_{t>0}$ is a contractive semigroup in $L^2(\mathbb R^d)$. Hence, for every $f\in L^2 (\mathbb R^d)$ and $m\in \mathbb N$, $F_f$ is $L^2(\mathbb R^d)$-Bochner integrable in $(1/m,\infty)$. We define
\[
\mathcal L_{V,\textup{heat}}^s(f)=\frac{1}{\Gamma (-s)}\lim_{m\rightarrow \infty}\int_{1/m}^\infty\frac{T_t^Vf-f}{t^{s+1}}dt,\quad \mbox{ in }L^2(\mathbb R^d),
\]
for each $f\in D(\mathcal L_{V,\textup{heat}}^s)$, being
\[
D(\mathcal L_{V,\textup{heat}}^s)=\left\{f\in L^2 (\mathbb R^d): \lim_{m\rightarrow \infty}\int_{1/m}^\infty\frac{T_t^Vf-f}{t^{s+1}}dt\mbox{ exists in }L^2(\mathbb R^d)\right\}.
\]
In \cite[Proposition 2.4]{BDFQ} it was proved that if $f\in D(\mathcal L_V^s)$, then $f\in D(\mathcal L_{V,\textup{heat}}^s)$ and $\mathcal L_V^sf=\mathcal L_{V,\textup{heat}}^sf$. Furthermore, if $f\in D(\mathcal L_V^\beta)$, with $s<\beta \leq 1$, then
\[
\mathcal L_V^sf=\frac{1}{\Gamma (-s)}\int_0^\infty \frac{T_t^Vf-f}{t^{s+1}}dt,
\]
where the integral is understood in the $L^2(\mathbb R^d)$-Bochner sense.

Let $\theta\in (0,1]$ and $f$ be a complex function defined in $\mathbb{R}^d$. We say that $f$ is in ${\textup{Lip}}_{{\rm loc},x}^\theta(\mathbb{R}^d)$, with $x\in \mathbb{R}^d$, when
\[
\|f\|_{{\textup{Lip}}_{{\rm loc},x}^\theta(\mathbb{R}^d)}:=\sup_{y\in B(x,1)}\frac{|f(x)-f(y)|}{|x-y|^\theta}<\infty.
\]
The function $f$ is said to be in ${\textup{Lip}}^\theta_{\rm loc}(\mathbb{R}^d)$ when $f\in {\textup{Lip}}^\theta_{{\rm loc},x}(\mathbb{R}^d)$, for every $x\in \mathbb{R}^d$.  We say that $f$ is in ${\textup{Lip}}^\theta_{{\rm loc},{\rm uni}}(\mathbb{R}^d)$ when
\[
\sup_{x\in \mathbb{R}^d}\sup_{y\in B(x,1)}\frac{|f(x)-f(y)|}{|x-y|^\theta}<\infty.
\]

Consider $d\geq 3$, $q>d/2$ and $V\in \RH_q$. In Section~\ref{sec: 2} (see Remark \ref{rem}) for every $\theta \in (0,1]$, we extend the operator $\mathcal L_V^s$, $s\in (0,\min\{\frac{\theta }{2},\frac{\delta}{2}\})$, being $\delta =2-d/q$, on the space of functions in ${\textup{Lip}}_{\textup{loc}}^\theta (\mathbb R^d)\cap L^1_0(\mathbb R^d)$, where $L^1_\sigma(\mathbb R^d)$, $\sigma \in \mathbb R$, represents the space of measurable functions on $\mathbb R^d$ such that $\int_{\mathbb R^d}|f(y)|(1+|y|)^{-d-\sigma}dy<\infty$. 

Likewise, in Remark \ref{rem} we extend the definition of $\mathcal L_V^{-s}$, $s\in (0,\min\{\frac{\varepsilon }{2},\frac{1}{4}\})$ and $\varepsilon \in (0,d)$, on the space ${\textup{Lip}}_{\textup{loc}}^\theta (\mathbb R^d)\cap L^1_{-\varepsilon}(\mathbb R^d)$. 

By reading the proof of Theorem \ref{ThB}  we can see that 
\begin{align}\label{1.5}
\lim_{s\rightarrow 0^+}\frac{1}{s}(\mathcal L_V^sf(x)-f(x))&=\int_{B(x,1)}(f(x)-f(y))\int_0^\infty \frac{T_t^V(x,y)}{t}dtdy\nonumber\\
&\hspace{-2cm} {-}\int_{\mathbb R^d\setminus B(x,1)}f(y)\int_0^\infty \frac{T_t^V(x,y)}{t}dtdy-K(x)f(x),\quad x\in \mathbb R^d, 
\end{align}
provided that $f\in {\textup{Lip}}^\theta_{\rm loc} (\mathbb R^d)\cap L^1_0(\mathbb R^d)$, for some $\theta \in (0,1]$. We extend the definition of $\log \mathcal L_V$ by defining $(\mathbb{L}\textup{og}\,\mathcal{L}_V)f$, for $f\in {\textup{Lip}}_{\rm loc}^\theta (\mathbb R^d)\cap L^1_0(\mathbb R^d)$, with $\theta \in (0,1]$, by using the right-hand side of \eqref{1.5}.

Moreover, Proposition \ref{propoY} and Proposition~\ref{Prop2.6} suggest us to define, for every $k\in \mathbb{N}$, the operator
\begin{equation}\label{Logk}
(\mathbb{L}\textup{og}^k\mathcal L_V)(f)(x):=\lim_{s\rightarrow 0^+}\partial _s^k\mathcal L_V^s(f)(x),\quad x\in \mathbb R^d,
\end{equation}
when $f\in {\textup{Lip}}^\theta _{\rm loc}(\mathbb R^d)\cap L^1_0(\mathbb R^d)$, obtaining that, when $f\in {\textup{Lip}}_{\rm loc}^\theta (\mathbb R^d)\cap L^1_{-\varepsilon}(\mathbb R^d)$, for some $\varepsilon \in (0,d)$, 
\[
(\mathbb{L}\textup{og}^k\mathcal L_V)(f)(x)=(-1)^k\lim_{s\rightarrow 0^+}\partial _s^k\mathcal L_V^{-s}(f)(x),\quad x\in \mathbb R^d.
\]
Note that $\mathbb{L}\textup{og}^k\,\mathcal{L}_V=\mathbb{L}\textup{og}\,\mathcal{L}_V$, when $k=1$.

We also see (Proposition~\ref{Ccreglog}) that each function $f\in C_c^\infty (\mathbb R^d)$ belongs to $D(\log ^k\mathcal L_V)$ and
\[
(\mathbb{L}\textup{og}^k\mathcal L_V)(f)(x)=(\log ^k\mathcal L_V)(f)(x),\quad \text{ a.e. }x\in \mathbb R^d.
\]

We extend the results in Theorem \ref{ThC} and Theorem \ref{Th1.1} in the following way. 

\begin{thm}\label{Th1.2}
Let $d\geq 3$, $q>d/2$, $V\in \RH_q$, and $0<\theta \leq 1$. 
\begin{enumerate}[label=\textnormal{(\alph*)}]
    \item \label{itm: Th1.2-a} Suppose that $f\in C_c(\mathbb R^d)\cap {\textup{Lip}}\,^\theta _{\textup{loc},\textup{uni}}(\mathbb R^d)$. Then, for every $1<p<\infty$ and $n\in \mathbb N$ we have that
\begin{equation}\label{1.10}
\lim_{s\rightarrow 0^+}\left(\mathcal L_V^sf-f-\sum_{k=1}^{n-1}\frac{s^k}{k!}(\mathbb{L}\textup{og}^k\mathcal L_V)f\right)\frac{n!}{s^n}=(\mathbb{L}\textup{og}^n\mathcal L_V)f,\quad \mbox{ in }L^p(\mathbb R^d).
\end{equation}
\item \label{itm: Th1.2-b} Suppose that $f\in L^1(\mathbb R^d)\cap L^\infty (\mathbb R^d)\cap {\textup{Lip}}\,^\theta _{\textup{loc},\textup{uni}}(\mathbb R^d)$. Then, for each $n\in \mathbb N$ and every compact set $\mathbb K\subset \mathbb R^d$, 
\begin{equation}\label{1.11}
\lim_{s\rightarrow 0^+}\left(\mathcal L_V^sf-f-\sum_{k=1}^{n-1}\frac{s^k}{k!}(\mathbb{L}\textup{og}^k\mathcal L_V)f\right)\frac{n!}{s^n}=(\mathbb{L}\textup{og}^n\mathcal L_V)f,\quad \mbox{ in }L^\infty(\mathbb K).
\end{equation}
\end{enumerate}
\end{thm}

We now state the results for the negative power of $\mathcal L_V$.

\begin{thm}\label{Th1.3} Let $d\geq 3$, $q>d/2$, $V\in \RH_q$, and $0<\theta \leq 1$.
\begin{enumerate}[label=\textnormal{(\alph*)}]
    \item \label{itm: Th1.3-a} Suppose that $f\in C_c(\mathbb R^d)\cap {\textup{Lip}}\,^\theta _{\textup{loc},\textup{uni}}(\mathbb R^d)$. Then, for every $1<p\leq \infty$ and $n\in \mathbb N$, we have that
\[
\lim_{s\rightarrow 0^+}\left(\mathcal L_V^{-s}f-f-\sum_{k=1}^{n-1}\frac{(-1)^ks^k}{k!}(\mathbb{L}\textup{og}^k\mathcal L_V)f\right)\frac{(-1)^nn!}{s^n}=(\mathbb{L}\textup{og}^n\mathcal L_V)f,\; \mbox{ in }L^p(\mathbb R^d).
\]
\item \label{itm: Th1.3-b} Suppose that $f\in L^1(\mathbb R^d)\cap L^\infty (\mathbb R^d)\cap {\textup{Lip}}\,^\theta _{\textup{loc},\textup{uni}}(\mathbb R^d)$. Then, for each $n\in \mathbb N$ and every compact set $\mathbb K\subset \mathbb R^d$, 
\[
\lim_{s\rightarrow 0^+}\left(\mathcal L_V^{-s}f-f-\sum_{k=1}^{n-1}\frac{(-1)^ks^k}{k!}(\mathbb{L}\textup{og}^k\mathcal L_V)f\right)\frac{(-1)^nn!}{s^n}=(\mathbb{L}\textup{og}^n\mathcal L_V)f,\; \mbox{ in }L^\infty(\mathbb K).
\]
\end{enumerate}
\end{thm}

Let $f\in C_c(\mathbb R^d)$ and $s\in (0,1)$. We define $u_{f,s}=\mathcal L_V^{-s}f$. Since $C_c(\mathbb R^d)\subset D(\mathcal L_V^{-s})$ (see Corollary \ref{cor1}), by using \cite[Theorem 3(iv), p. 343]{Yo}, 
$u_{f,s}\in D(\mathcal L_V^{-s})$ and $\mathcal L_V^su_{f,s}=f$.

Suppose that $R>0$ and ${\rm supp}\,f\subset B(0,R)$. Then, by using Corollary \ref{cor1} and \eqref{1.6} we get 
\[
|u_{f,s}(x)|\leq C\int_{B(0,R)}\frac{|f(y)|}{|x-y|^{d-2s}}dy\leq \frac{C}{|x|^{d-2s}}\int_{B(0,R)}|f(y)|dy,\quad |x|\geq 2R.
\]
Hence, $\lim_{x\rightarrow \infty}u_{f,s}(x)=0$.

From Theorem \ref{Th1.3} we deduce the following result.
\begin{cor}
    Let $d\geq 3$, $q>d/2$, $V\in \RH_q$, and $s\in (0,1]$. Suppose that $f\in C_c(\mathbb R^d)\cap {\textup{Lip}}\,^\theta _{\textup{loc}, \textup{uni}}(\mathbb R^d)$. Then, for every $s\in (0,1)$, $u_{f,s}=\mathcal L
    _V^{-s}f$, is a solution of
    \[
    \left\{\begin{array}{r}
        \mathcal L_V^su=f\\
        \displaystyle \lim_{x\rightarrow \infty }u(x)=0 
    \end{array}
    \right.
    \]
    and, for each $n\in \mathbb N$,
    \[
    u_{f,s}=f+\sum_{j=1}^n\frac{(-1)^js^j}{j!}\mathbb{L}\textup{og}^j\mathcal{L}_Vf+o(s^n),\mbox{ as }s\rightarrow 0^+,
    \]
    in $L^p(\mathbb R^d)$ when $1<p<\infty$, and in $L^\infty(\mathbb K)$, for every compact subset $\mathbb K\subset \mathbb R^d$.
\end{cor}

In the next section, we present the main results concerning fractional powers $\mathcal{L}_V^s$. The proofs of Theorems \ref{Th1.1}, \ref{Th1.2}, and \ref{Th1.3} are provided in the subsequent sections. Throughout this paper, $C$ and $c$ denote positive constants whose values may vary from line to line. 

\section{About fractional powers of \texorpdfstring{$\mathcal{L}_V$}{LV}}~\label{sec: 2}
As previously mentioned, in \cite[Proposition 2.4]{BDFQ} the authors proved that, given $s\in (0,1)$ and $f\in L^2(\mathbb R^d)$, if $f\in D(\mathcal L_V^\beta)$ for some $s<\beta \leq 1$, then
\begin{equation}\label{converL2}
\mathcal L_V^sf=\frac{1}{\Gamma (-s)}\int_0^\infty \frac{T_t^Vf-f}{t^{s+1}}dt,
\end{equation}
where the integral is understood in the $L^2(\mathbb R^d)$-Bochner sense, and from (\ref{converL2}) it follows that, for almost all $x\in \mathbb{R}^d$,
\[
(\mathcal{L}_V^sf)(x)=\frac{1}{\Gamma(-s)}\int_0^\infty\frac{T_t^V(f)(x)-f(x)}{t^{s+1}}dt.
\]
We recall that $C_c^\infty(\mathbb{R}^d)\subset D(\mathcal{L}_V)\subset D(\mathcal{L}_V^s)$, $s\in (0,1)$.

We first study the behavior of the positive power $\mathcal{L}_V^s$ on Lipschitz functions.

\begin{prop}\label{PropLips} 
Let $d\geq 3$, $q>d/2$, and $V\in {\rm RH}_q$. For every $f\in {\textup{Lip}}_\textup{loc}^\theta (\mathbb R^d)\cap L^1_0(\mathbb R^d)$, with $\theta \in (0,1)$, 
\begin{equation}\label{limLips}
\lim_{s\rightarrow 0^+}\frac{1}{\Gamma (-s)}\int_0^\infty \frac{T_t^V(f)(x)-f(x)}{t^{s+1}}dt=f(x),\quad x\in \mathbb R^d,
\end{equation}
being the integral absolutely convergent for each $0<s<\min\left\{\tfrac{\theta}{2},\tfrac{\delta}{2}\right\}$, where \linebreak ${\delta =2-d/q}$. 
\end{prop}
\begin{proof}
Assume that $f\in {\textup{Lip}}_\textup{loc}^\theta (\mathbb R^d)\cap L^1_0(\mathbb R^d)$. Observe that, since $f\in L^1_0(\mathbb R^d)$, using \eqref{1.6} we obtain that 
\begin{align*}
|T_t^Vf(x)|&\leq \frac{C}{t^{d/2}}\int_{\mathbb R^d}e^{-\frac{|x-y|^2}{4t}}|f(y)|dy\\
&\leq C\left(\frac{1}{t^{d/2}}\int_{|y|\leq 2|x|}|f(y)|dy+\int_{|y|>2|x|}\frac{|f(y)|}{(1+|x-y|)^d}dy\right)\\
 &\leq C\left(\frac{(1+|x|)^d}{t^{d/2}}\int_{\mathbb R^d}\frac{|f(y)|}{(1+|y|)^d}dy+\int_{|y|>2|x|}\frac{|f(y)|}{(1+|y|)^d}dy\right)\\
 & \leq C(1+|x|)^d\left(\frac{1}{t^{d/2}}+1\right)\int_{\mathbb R^d}\frac{|f(y)|}{(1+|y|)^d}dy<\infty,\quad x\in \mathbb R^d\mbox{ and }t>0.
\end{align*}
On the other hand, since $T_t(1)(x)=1$, $x\in \mathbb R^d$ and $t>0$, we can write
\begin{align}\label{sumI}
   \int_0^\infty \frac{T_t^V(f)(x)-f(x)}{t^{s+1}}dt&\nonumber\\
   &\hspace{-3cm}=\int_0^\infty \frac{T_t^V(f)(x)-T_t^V(1)(x)f(x)}{t^{s+1}}dt+f(x)\int_0^{\rho (x)^2}\frac{T_t^V(1)(x)-T_t(1)(x)}{t^{s+1}}dt\nonumber\\
    &\hspace{-3cm}\quad +f(x)\int_{\rho (x)^2}^\infty \frac{T_t^V(1)(x)}{t^{s+1}}dt-f(x)\int_{\rho (x)^2}^\infty \frac{dt}{t^{s+1}}\nonumber\\
    &\hspace{-3cm}= \int_0^\infty\int_{\mathbb R^d}\big(f(y)-f(x)\mathcal X_{B(x,1)}(y)\big)\frac{T_t^V(x,y)}{t^{s+1}}dydt\nonumber\\
    &\hspace{-3cm}\quad -f(x) \int_0^\infty\int_{\mathbb R^d\setminus B(x,1)}\frac{T_t^V(x,y)}{t^{s+1}}dydt\nonumber\\
    &\hspace{-3cm}\quad +f(x)\int_0^{\rho (x)^2}\int_{\mathbb R^d}\frac{T_t^V(x,y)-T_t(x-y)}{t^{s+1}}dydt\nonumber\\
    &\hspace{-3cm}\quad +f(x) \int_{\rho (x)^2}^\infty\int_{\mathbb R^d}\frac{T_t^V(x,y)}{t^{s+1}}dydt-f(x)\frac{\rho (x)^{-2s}}{s}\nonumber\\
     &\hspace{-3cm}= \int_0^\infty\int_{\mathbb R^d}\big(f(y)-f(x)\mathcal X_{B(x,1)}(y)\big)\frac{T_t^V(x,y)}{t^{s+1}}dydt\nonumber\\
    &\hspace{-3cm}\quad +f(x)\int_0^{\rho (x)^2}\int_{\mathbb R^d}\frac{T_t^V(x,y)-T_t(x-y)}{t^{s+1}}dydt\nonumber\\
    &\hspace{-3cm}\quad +f(x)\left(\int_{\rho (x)^2}^\infty\int_{B(x,1)}-\int_0^{\rho (x)^2}\int_{\mathbb R^d\setminus B(x,1)} \right)\frac{T_t^V(x,y)}{t^{s+1}}dydt-f(x)\frac{\rho (x)^{-2s}}{s}\nonumber\\
&\hspace{-3cm}=:\sum_{j=1}^4I_j(x,s),\quad x\in \mathbb R^d\mbox{ and }s\in (0,1).
\end{align}

It is clear that 
\begin{equation}\label{limI4}
   \lim_{s\rightarrow 0^+}\frac{1}{\Gamma (-s)}I_4(x,s)=\lim_{s\rightarrow 0^+}\frac{f(x)\rho (x)^{-2s}}{\Gamma (1-s)}=f(x),\quad x\in \mathbb R^d. 
\end{equation}

We study now $I_j$, $j=1,2,3$. Using \eqref{1.6} we obtain
\begin{align*}
   |I_1(x,s)|&\leq C\left(\int_{|x-y|\leq 1}|f(y)-f(x)|\int_0^ \infty \frac{e^{-\frac{|x-y|^2}{4t}}}{t^{d/2+s+1}}dtdy\right.\nonumber\\
  &\quad \left.+\int_{|x-y|>1}|f(y)|\int_0^ \infty \frac{e^{-\frac{|x-y|^2}{4t}}}{t^{d/2+s+1}}dtdy\right)\\
  &\leq C\left(\int_{|x-y|\leq 1}\frac{|f(y)-f(x)|}{|x-y|^{d+2s}}dy+\int_{|x-y|>1}\frac{|f(y)|}{|x-y|^{d+2s}}dy\right)\\
  &\leq C\left(\int_{|x-y|\leq 1}|x-y|^{\theta -d-2s}dy+\int_{\mathbb R^d}\frac{|f(y)|}{(1+|x-y|)^d}dy\right)\\
  &\leq C\left(\frac{1}{\theta -2s}+(1+|x|)^d\right),\quad x\in \mathbb R^d,
\end{align*}
provided that $0<s<\tfrac{\theta}{2}$. Then
\begin{equation}\label{1.7}
\lim_{s\rightarrow 0^+}\frac{1}{\Gamma (-s)}I_1(x,s)=0,\quad x\in \mathbb R^d.
\end{equation}

Now, according to \cite[Proposition 2.16]{DzZ3}, there exists a nonnegative function $\phi$ in the Schwartz class $\mathcal S(\mathbb R^d)$ such that 
\begin{equation}\label{TtVTt}
|T_t^V(x,y)-T_t(x-y)|\leq \left(\frac{\sqrt{t}}{\rho (x)}\right)^\delta \phi _t(x-y),\quad x,y\in \mathbb R^d\mbox{ and }t>0,
\end{equation}
where $\delta =2-d/q$, and $\phi_t(z)=t^{-d/2}\phi (z/\sqrt{t})$, $z\in \mathbb R^d$ and $t>0$.
Using this estimate we get
\begin{align*}
|I_2(x,s)|&\leq C|f(x)|\int_0^{\rho (x)^2}\frac{1}{t^{s+1}}\int_{\mathbb R^d}\left(\frac{\sqrt{t}}{\rho (x)}\right)^\delta |\phi _t(x-y)|dydt\nonumber\\
&=C\frac{|f(x)|}{\rho (x)^\delta}\int_0^{\rho (x)^2}\frac{dt}{t^{s -\tfrac{\delta}{2}+1}}=C\frac{|f(x)|}{\rho (x)^{2s}(\tfrac{\delta}{2}-s)},\quad x\in \mathbb R^d,
\end{align*}
provided that $0<s<\tfrac{\delta}{2}$. 
Thus, we deduce that
\begin{equation}\label{1.8}
\lim_{s\rightarrow 0^+}\frac{1}{\Gamma (-s)}I_2(x,s)=0,\quad x\in \mathbb R^d.
\end{equation}

Finally, using \eqref{acotTVimpr} for $N=1$ we get
\begin{align*}
|I_3(x,s)|&\leq C|f(x)|\rho (x)\left(\int_{\rho (x)^2}^\infty\int_{|x-y|<1}+\int_0^{\rho (x)^2}\int_{|x-y|\geq 1}\right)\frac{e^{-c\frac{|x-y|^2}{t}}}{t^{d/2+s+3/2}}dydt\\
&\leq  C|f(x)|\rho (x)\left(\int_{\rho (x)^2}^\infty \int_0^1
+\int_0^\infty \int_1^\infty \right)\frac{e^{-c\frac{r^2}{t}}r^{d-1}}{t^{d/2+s+3/2}}drdt\\
&\leq C|f(x)|\rho (x)\left(\int_{\rho (x)^2}^\infty \frac{dt}{t^{s+3/2}}\int_0^\infty e^{-u}u^{d/2-1}du+\int_1^\infty \frac{dr}{r^{2s+2}}\right)\\
&\leq C|f(x)|\rho (x)\left(1+\rho (x)^{-2s-1}\right),\quad x\in \mathbb R^d\mbox{ and }s\in (0,1).
\end{align*}
It follows that
\begin{equation}\label{1.9}
\lim_{s\rightarrow 0^+}\frac{1}{\Gamma (-s)}I_3(x,s)=0,\quad x\in \mathbb R^d.
\end{equation}
From \eqref{sumI}, \eqref{limI4}, \eqref{1.7}, \eqref{1.8} and \eqref{1.9} we deduce that \eqref{limLips} holds.   
\end{proof}

We now establish some properties for the negative power $\mathcal{L}_V^{-s}$, $s\in (0,1)$, of $\mathcal{L}_V$.

\begin{prop}\label{Ccreg} If $f\in L^2(\mathbb{R}^d)$ is  such that $\int_0^\infty t^{s-1}\|T_t^Vf\|_{L^2(\mathbb{R}^d)}dt<\infty$ with $s\in (0,1)$, then $f\in D(\mathcal{L}_V^{-s})$ and 
$\mathcal{L}_V^{-s}f=\frac{1}{\Gamma(s)}\int_0^\infty t^{s-1}T_t^V(f)dt$, where the integral is understood in the $L^2(\mathbb{R}^d)$-Bochner sense.
\end{prop}
\begin{proof} Suppose that $f\in L^2(\mathbb R^d)$ such that $\int_0^\infty t^{s -1}\|T_t^Vf\|_{L^2(\mathbb R^d)}dt<\infty$, with $s \in (0,1)$. We are going to see that $f\in D(\mathcal L_V^{-s})$, i.e., $\int_0^\infty \lambda^{-s}dE_V(\lambda)f\in L^2(\mathbb R^d)$ or, equivalently, $\int_0^\infty \lambda^{-2s}d\mu_{f,f}(\lambda)<\infty$.

For every $n\in \mathbb N$ we define $J_n=\int_{1/n}^\infty \lambda^{-s }dE_V(\lambda )f$. We have that
\[
\|J_n\|_{L^2(\mathbb R^d)}=\left(\int_{\frac{1}{n}}^\infty \lambda^{-2s }d\mu _{f,f,}(\lambda)\right)^{1/2}\leq n^{2s}\|f\|_{L^2(\mathbb R^d)},\quad n\in \mathbb N.
\]
Let $g\in L^2(\mathbb R^d)$. We get, for $n,m\in \mathbb N$, $m<n$, 
\begin{align}
\int_{\mathbb R^d}(J_n-J_m)(x)g(x)dx&=\int_{\mathbb R^d}\left(\int_{\frac{1}{n}}^{\frac{1}{m}}\lambda^{-s }dE_V(\lambda) f\right)(x)g(x)dx\nonumber\\
&\hspace{-4cm}=\int_{\frac{1}{n}}^{\frac{1}{m}}\lambda^{-s }d\mu_{f,g}(\lambda)=\int_{\frac{1}{n}}^{\frac{1}{m}}\int_0^\infty \frac{e^{-t\lambda}t^{s -1}}{\Gamma (s)}dtd\mu_{f,g}(\lambda)=\int_0^\infty \frac{t^{s -1}}{\Gamma (s)}\int_{\frac{1}{n}}^{\frac{1}{m}} e^{-t\lambda}d\mu_{f,g}(\lambda)\nonumber\\
&\hspace{-4cm}=\int_0^\infty \frac{t^{s -1}}{\Gamma (s)}\int_{\mathbb R^d}g(x)\left(\int_{\frac{1}{n}}^{\frac{1}{m}}e^{-t\lambda}dE_V(\lambda )f\right)(x)dxdt\label{(1)}\\
&\hspace{-4cm}=\int_{\mathbb R^d}g(x)\int_0^\infty \frac{t^{s -1}}{\Gamma (s)}\left(\int_{\frac{1}{n}}^{\frac{1}{m}}e^{-t\lambda}dE_V(\lambda )f\right)(x)dtdx.\label{(2)}
\end{align}
Note that
\begin{align*}
\int_0^\infty \frac{t^{s -1}}{\Gamma (s)}\int_{\frac{1}{n}}^{\frac{1}{m}} e^{-t\lambda}d|\mu_{f,g}|(\lambda)&\leq \int_0^\infty e^{-\frac{t}{n}}t^{s -1}dt\int_0^\infty d|\mu_{f,g}|(\lambda)\\
&\leq C\|f\|_{L^2(\mathbb R^d)}\|g\|_{L^2(\mathbb R^d)},\quad n,m\in \mathbb N,\,m<n,
\end{align*}
and
\begin{align*}
    \int_0^\infty t^{s -1}\int_{\mathbb R^d}|g(x)|&\left|\left(\int_{\frac{1}{n}}^{\frac{1}{m}}e^{-t\lambda}dE_V(\lambda )f\right)(x)\right|dxdt&\\
    &\leq  \|g\|_{L^2(\mathbb R^d)}\int_0^\infty t^{s-1}\left\|\int_{\frac{1}{n}}^{\frac{1}{m}}e^{-t\lambda}dE_V(\lambda )f\right\|_{L^2(\mathbb R^d)}dt\\
    &\leq \|g\|_{L^2(\mathbb R^d)}\int_0^\infty t^{s -1}\left(\int_{\frac{1}{n}}^{\frac{1}{m}}e^{-2t\lambda}d\mu _{f,f}(\lambda)\right)^{1/2}dt\\
    &\leq \|f\|_{L^2(\mathbb R^d)}\|g\|_{L^2(\mathbb R^d)}\int_0^\infty t^{s -1}e^{-\frac{t}{n}}dt,\quad n,m\in \mathbb N,\,m<n,
\end{align*}
and the interchanges of order of integration in \eqref{(1)} and \eqref{(2)} are justified. 

We deduce that, for every $n,m\in \mathbb N$, with $m<n$,
\[
(J_n-J_m)(x)=\int_0^\infty \frac{t^{s -1}}{\Gamma (s)}\left(\int_{\frac{1}{n}}^{\frac{1}{m}} e^{-t\lambda}dE_V(\lambda)f\right)(x)dt,\quad \text{ a.e. }x\in \mathbb R^d.
\]
Then,
\begin{align*}
\|J_n-J_m\|_{L^2(\mathbb R^d)}&\leq \int_0^\infty \frac{t^{s -1}}{\Gamma (s)}\left\|\int_{\frac{1}{n}}^{\frac{1}{m}} e^{-t\lambda}dE_V(\lambda)f\right\|_{L^2(\mathbb R^d)}dt\\
&=\frac{1}{\Gamma (s)}\int_0^\infty t^{s-1}\left(\int_{\frac{1}{n}}^{\frac{1}{m}} e^{-2t\lambda}d\mu _{f,f}(\lambda)\right)^{1/2}dt,\quad n,m\in \mathbb N,\,m<n.
\end{align*}

On the other hand, we can write
\[
\int_0^\infty t^{s -1}\|T_tf\|_{L^2(\mathbb R^d)}dt=\int_0^\infty t^{s -1}\left(\int_0^\infty e^{-2t\lambda}d\mu _{f,f}(\lambda)\right)^{1/2}dt<\infty,
\]
and, for each $t>0$,
\[
\int_0^{\frac{1}{n}}e^{-2t\lambda}d\mu _{f,f}(\lambda)\leq \int_0^{\frac{1}{n}}d\mu _{f,f}(\lambda)=\langle E_V(0,\tfrac{1}{n})f,f\rangle \longrightarrow 0,\mbox{ as }n\rightarrow \infty.
\]
Monotone convergence theorem leads to 
\[
\lim_{n\rightarrow \infty }\int_0^\infty t^{s -1}\left(\int_0^{\frac{1}{n}}e^{-2t\lambda}d\mu _{f,f}(\lambda)\right)^{1/2}dt=0.
\]
Hence, for every $\varepsilon >0$ there exists $n_0\in \mathbb N$ such that
\[
\|J_n-J_m\|_{L^2(\mathbb R^d)}<\varepsilon,\quad n,m\in \mathbb N,\,n_0<m<n.
\]
Thus we have proved that $\{J_n\}_{n\in \mathbb N}$ is a Cauchy sequence in $L^2(\mathbb R^d)$. Then, there exists $F\in L^2(\mathbb R^d)$ for which $J_n\longrightarrow F$ in $L^2(\mathbb R^d)$. We conclude that $f\in D(\mathcal L_V^{-s})$ and that
\[
\mathcal L_V^{-s}f=\frac{1}{\Gamma (s )}\int_0^\infty t^{s -1}T_t^Vfdt.\qedhere
\]
\end{proof}

Note that from Proposition \ref{Ccreg} we deduce that, for every $s\in (0,1)$,
\[
(\mathcal{L}_V^{-s}f)(x)=\frac{1}{\Gamma(s)}\int_0^\infty t^{s-1}T_t^V(f)(x)dt, \quad a.e.\quad x\in \mathbb{R}^d,
\]
provided that $f\in L^2(\mathbb{R}^d)$ and $\int_0^\infty t^{s-1}\|T_tf\|_{L^2(\mathbb{R}^d)}dt<\infty$.

\begin{cor}\label{cor1} Let $d\geq 3$, $q>d/2$, and $V\in {\rm RH}_q$. If $f\in C_c(\mathbb{R}^d)$ and $s\in (0,1)$, then $f\in D(\mathcal{L}_V^{-s})$, and
\[
\mathcal L_V^{-s}f=\frac{1}{\Gamma(s)}\int_0^\infty t^{s-1}T_t^V(f)dt,
\]
where the integral is understood in the $L^2(\mathbb R^d)$-Bochner sense.
\end{cor}

\begin{proof}
 Let $f\in C_c(\mathbb{R}^d)$  and $s\in (0,1)$. We choose $R>0$ such that $\supp(f)\subset B(0,R)$. According to (\ref{1.6}) we have that
 \begin{equation}\label{ZZ1}
|T_t^V(f)(x)|\le Ct^{-d/2}, \quad x\in \mathbb{R}^d\quad and \quad t>0.
 \end{equation}
 By (\ref{acotTVimpr}) and \cite[Lemma 1.4(b)]{Sh1} we get
 \begin{align}\label{ZZ3}
 |T_t^V(f)(x)|&\le C\int_{B(0,R)}|f(y)|\frac{e^{-c\frac{|x-y|^2}{t}}}{t^{d/2}}\frac{\rho(y)}{\sqrt{t}}dy\nonumber\\
 &\le C \frac{e^{-c|x|^2/t}}{t^{(d+1)/2}}\int_{B(0,R)}|f(y)|\rho(y)dy\nonumber\\
 &\le C \frac{e^{-c|x|^2/t}}{t^{(d+1)/2}},\quad x\notin B(0,2R)\quad and\quad t>0.
 \end{align}
 It follows from (\ref{ZZ1}) and (\ref{ZZ3}) that
\begin{align}\label{ZZ2} 
\|T_t^Vf\|_{L^2(\mathbb{R}^d)}&\le \left(\int_{B(0.2R)}|T_t^V(f)(x)|^2dx+\int_{\mathbb{R}^d\setminus B(0,2R)}|T_t^V(f)(x)|^2dx\right)^{1/2}\nonumber\\
&\le C\left(t^{-d}+ \int_{\mathbb{R}^d\setminus B(0,2R)}\frac{e^{-c|x|^2/t}}{t^{d+1}}dx\right)^{1/2}\nonumber\\
&\le C\left(t^{-d/2}+t^{-(d+2)/4}\right)\le Ct^{-(d+2)/4},\quad t\ge 1.
\end{align}
By using (\ref{ZZ1}) and (\ref{ZZ2}) we deduce that
\[
\int_0^\infty t^{s-1}\|T_t^Vf\|_{L^2(\mathbb{R}^d)}dt\le C\left(\int_0^1 t^{s-1}dt+\int_1^\infty t^{s-1-\frac{d+2}{4}}dt\right)<\infty.
\]
According to Proposition \ref{Ccreg} we conclude that $f\in D(\mathcal{L}_V^{-s})$.
\end{proof}

\begin{prop}\label{PropLip-s}
Let $d\geq 3$, $q>d/2$, and $V\in {\rm RH}_q$. For every $f\in {\textup{Lip}}_\textup{loc}^\theta (\mathbb R^d)\cap L^1_{-\varepsilon}(\mathbb R^d)$, with $\theta \in (0,1)$ and $\varepsilon \in (0,d)$, 
\begin{equation}\label{limLip}
\lim_{s\rightarrow 0^+}\frac{1}{\Gamma (s)}\int_0^\infty T_t^V(f)(x)t^{s-1}dt=f(x),\quad x\in \mathbb R^d,
\end{equation}
being the integral absolutely convergent for each $0<s<\min\left\{\tfrac{\varepsilon}{2}, \tfrac14\right\}$. 
\end{prop}
\begin{proof}
Consider $f\in {\textup{Lip}}^\theta _{\rm loc}(\mathbb R^d)\cap L^1_{-\varepsilon}(\mathbb R^d)$, for some $\varepsilon >0$. We can write
\begin{align}\label{sumR}
\int_0^\infty &T_t^V(f)(x)t^{s-1}dt\nonumber\\
   &=\int_0^\infty(T_t^V(f)(x)-f(x)T_t^V(1)(x))t^{s-1}dt+f(x)\int_{\rho (x)^2}^\infty T_t^V(1)(x)t^{s-1}dt\nonumber\\
    &\quad +f(x)\left(\int_0^{\rho (x)^2} (T_t^V(1)(x)-T_t(1)(x))t^{s-1}dt+\int_0^{\rho (x)^2}T_t(1)(x)t^{s-1}dt\right)\nonumber\\
     &=\int_0^\infty\int_{\mathbb R^d}\big(f(y)-f(x)\mathcal X_{B(x,1)}(y)\big)T_t^V(x,y)t^{s-1}dydt\nonumber\\
      &\quad +f(x)\int_0^{\rho (x)^2} \int_{\mathbb R^d}(T_t^V(x,y)-T_t(x-y))t^{s-1}dt\nonumber\\
      &\quad +f(x)\left(\int_{\rho (x)^2}^\infty\int_{B(x,1)}-\int_0^{\rho (x)^2}\int_{\mathbb R^d\setminus B(x,1)}\right)T_t^V(x,y)t^{s-1}dt+\frac{f(x)\rho (x)^{2s}}{s}\nonumber\\   
    &=:\sum_{j=1}^4R_j(x,s),\quad x\in \mathbb R^d\mbox{ and }s\in (0,\min\left\{\tfrac{\varepsilon }{2},\tfrac{1}{4}\right\}).
\end{align}
Indeed, by proceeding as in the proof of Proposition \ref{PropLips} we get 
\begin{align*}
|R_1(x,s)|&\leq C\left(\int_{|x-y|\leq 1}\frac{|f(x)-f(y)|}{|x-y|^{d-2s}}dy+\int_{|x-y|>1}\frac{|f(y)|}{|x-y|^{d-2s}}dy\right)\\
&\leq C\left(\int_{|x-y|\leq 1}|x-y|^{\theta -d+2s}dy+\int_{\mathbb R^d}\frac{|f(y)|}{(1+|x-y|)^{d-\varepsilon}}\right)\\
&\leq C\left(\frac{1}{\theta }+(1+|x|)^{d-\varepsilon}\right),\quad x\in \mathbb R^d\mbox{ and }s\in \left(0,\tfrac{\varepsilon}{2}\right).
\end{align*}
Also, considering \eqref{TtVTt}, we deduce that, for each $x\in \mathbb R^d$ and $s\in (0,1)$,
\begin{align*}
|R_2(x,s)|&\leq C\frac{|f(x)|}{\rho (x)^{\delta}}\int_0^{\rho (x)^2}t^{s+\tfrac{\delta}{2}-1}dt\leq C\frac{|f(x)|\rho (x)^{2s}}{2s+\delta}\\
&\leq C|f(x)|\rho (x)^{2s}\leq C|f(x)|(1+\rho(x))^2,
\end{align*}
and, from \eqref{acotTVimpr} with $N=1$, we obtain
\begin{align*}
|R_3(x,s)|&\leq C|f(x)|\rho (x)\left(\int_{\rho (x)^2}^\infty\int_{|x-y|<1}+\int_0^{\rho (x)^2}\int_{|x-y|\geq 1}\right)\frac{e^{-c\frac{|x-y|^2}{t}}}{t^{d/2-s+3/2}}dydt\\
&\leq  C|f(x)|\rho (x)\left(\int_{\rho (x)^2}^\infty \int_0^1
+\int_0^\infty \int_1^\infty \right)\frac{e^{-c\frac{r^2}{t}}r^{d-1}}{t^{d/2-s+3/2}}drdt\\
&\leq C|f(x)|\rho (x)\left(\int_{\rho (x)^2}^\infty \frac{dt}{t^{-s+3/2}}+\int_1^\infty \frac{dr}{r^{-2s+2}}\right)\\
&=C|f(x)|\frac{\rho (x)^{2s}+\rho (x)}{1-2s}\\\
&\leq C|f(x)|(\rho (x)^{2s}+\rho (x))\leq C|f(x)|(1+\rho(x)),\quad x\in \mathbb R^d\mbox{ and }s\in \left(0,\tfrac{1}{4}\right).
\end{align*}
Thus, the integral in \eqref{limLip} is absolutely convergent. We conclude that
\[
\lim_{s\rightarrow 0^+}\frac{1}{\Gamma (s)}\sum_{j=1}^4R_j(x,s)=\lim_{s\rightarrow 0^+}\sum_{j=1}^4\frac{sR_j(x,s)}{\Gamma (1+s)}=\lim_{s\rightarrow 0^+}f(x)\rho (x)^{2s}=f(x),\quad x\in \mathbb R^d.\qedhere
\]
\end{proof}

\begin{rem}\label{rem} 
Let $d\geq 3$, $q>d/2$, and $V\in {\rm RH}_q$. By virtue of Propositions \ref{PropLips} and \ref{PropLip-s}, for each $\theta \in (0,1)$ we define $\mathcal L_V^sf$, when $f\in {\textup{Lip}}_\textup{loc}^\theta (\mathbb R^d)\cap L^1_0(\mathbb R^d)$ and $s\in (0,\min\left\{\tfrac{\theta}{2},\tfrac{\delta}{2}\right\})$, where $\delta =2-d/q$, by 
\[
\mathcal L_V^sf(x)=\frac{1}{\Gamma (-s)}\int_0^\infty \frac{T_t^V(f)(x)-f(x)}{t^{s+1}}dt,\quad x\in \mathbb R^d, 
\]
and $\mathcal L_V^{-s}f$, when $f\in {\textup{Lip}}_\textup{loc}^\theta (\mathbb R^d)\cap L^1_{-\varepsilon}(\mathbb R^d)$, for some $\varepsilon \in (0,d)$ and $s\in (0,\min\left\{\tfrac{\varepsilon}{2},\tfrac{1}{4}\right\})$,
\[
\mathcal L_V^{-s}f(x)=\frac{1}{\Gamma (s)}\int_0^\infty T_t^Vf(x)t^{s-1}dt,\quad x\in \mathbb R^d.
\]
Moreover, for every $f\in {\textup{Lip}}_\textup{loc}^\theta (\mathbb R^d)\cap L^1_{-\varepsilon}(\mathbb R^d)$, for some $\varepsilon \in (0,d)$, we have that 
\[
\lim_{s\rightarrow 0^+}\mathcal L_V^sf(x)=\lim_{s\rightarrow 0^+}\mathcal L_V^{-s}f(x)=f(x),\quad x\in \mathbb R^d.
\]
\end{rem}
Next results involve the higher order logarithmic operators $\log^k\mathcal{L}_V$ of $\mathcal{L}_V$ with $k\in \mathbb{N}$.

\begin{prop}\label{derlog} Let $s_0\in (0,1]$ and $k_0\in \mathbb{N}$. Suppose that $f\in D(\mathcal{L}_V^{s_0})\cap D(\log^{k_0}\mathcal{L}_V)$. 
\begin{enumerate}[label=\textnormal{(\alph*)}]
\item \label{itm: derlog-a} For every $s\in (0,s_0)$ and $k\in \mathbb{N}$, $k\le k_0$, we have $f\in D(\mathcal{L}_V^s\log^k\mathcal{L}_V)\cap D((\log^k\mathcal{L}_V)\mathcal{L}_V^s)$ and $(\mathcal{L}_V^s\log^k\mathcal{L}_V)f=((\log^k\mathcal{L}_V)\mathcal{L}_V^s)f$.
 
\item \label{itm: derlog-b} For every $s\in (0,s_0)$ and $k\in \mathbb{N}$, $k\le k_0$, $\partial_s^k\mathcal{L}_V^sf=((\log^k\mathcal{L}_V)\mathcal{L}_V^s)f$ and $\lim_{s\to 0^+}\partial_s^k\mathcal{L}_V^sf=(\log^k\mathcal L_V)f$, in $L^2(\mathbb{R}^d)$, where $\partial_s$ is understood in $L^2(\mathbb{R}^d)$.
\end{enumerate}
\end{prop}

\begin{proof}
  Since $D(\mathcal{L}_V^{s_0})\subset D(\mathcal{L}_V^s)$, for every $s\in (0,s_0)$, and $D(\log^{k_0}\mathcal{L}_V)\subset D(\log^k\mathcal{L}_V)$, for each $k\in \mathbb{N}$, $k\le k_0$, assertions in \ref{itm: derlog-a} can be deduced from \cite[Theorem 3(iv), p. 343]{Yo}.

  We now prove \ref{itm: derlog-b}. Let $s\in (0,s_0)$. We consider the function 
  \begin{align*}
  \mathcal H(h):&=\frac{\mathcal{L}_V^{s+h}f-\mathcal{L}_V^sf}{h}-((\log\mathcal{L}_V)\mathcal {L}_V^s)f\\
  &= \left(\frac{\mathcal {L}_V^h-I}{h}-\log\mathcal{L}_V\right)\mathcal{L}_V^sf, \quad 0<h<s_0-s.
      \end{align*}
We have that
\[
\|\mathcal{H}(h)\|_{L^2(\mathbb {R}^d)}=\left(\int_0^\infty\left|\frac{\lambda^h-1}{h}-\log\lambda\right|^2\lambda^{2s}d\mu_{f,f}(\lambda)\right)^{1/2}, \quad 0<h<s_0-s.
\]
Mean value theorem implies that, for every $\lambda>0$ there exists $u\in (0,h) $ such that $\lambda^h-1=h\lambda^u\log\lambda$. Then 
\[
\Big|\frac{\lambda^h-1}{h}-\log\lambda\Big|\le |\log\lambda|(2+|\lambda|^{s_0-s}),
\]
for every $\lambda>0$, and $h\in (0,s_0-s)$.

Since $f\in D((\log\mathcal{L}_V)\mathcal{L}_V^{s_0})$, by applying dominated convergence theorem we deduce that $\lim_{h\to 0^+}\|\mathcal H(h)\|_{L^2(\mathbb{R}^d)}=0$. Thus we prove that
\[
\partial_s\mathcal{L}_V^sf=(\mathcal{L}_V^s\log\mathcal{L}_V)f.
\]
By iterating the argument we can prove that, for every $k\in \mathbb{N}$, $k\le k_0$,
\[
\partial_s^k\mathcal{L}_V^sf=(\mathcal{L}_V^s\log^k\mathcal{L}_V)f.
\]
Let $k\in \mathbb{N}$, $k\le k_0$. Since $(\log^k\mathcal{L}_V)f\in D(\mathcal{L}_V^{s_0})$ by proceeding as above we get
\[
\lim_{s\rightarrow 0^+}\big\|\partial_s^k\mathcal{L}_V^sf-(\log^k\mathcal{L}_V)f\big\|_{L^2(\mathbb{R}^d)}=\lim_{s\rightarrow 0^+}\left(\int_0^\infty|(\lambda^s-1)\log^k\lambda|^2d\mu_{f,f}(\lambda)\right)^{1/2}=0.\qedhere\]
\end{proof}
From Proposition \ref{derlog} we deduce the following property.

\begin{cor}\label{Ccreglog}
Let $d\ge 3$, $q>d/2$, $V\in \rm{RH}_q$ and $k\in \mathbb{N}$. For every $f\in C_c^\infty(\mathbb{R}^d)$,
\[
\lim_{s\to 0^+}\partial_s^k\mathcal{L}_V^sf=\log^k \mathcal{L}_V f, \quad in\quad L^2(\mathbb{R}^d).
\]
\begin{proof}
According to Corollary \ref{Ccreg}, $C_c^\infty(\mathbb{R}^d)\subset D(\mathcal{L}_V^{-s})$, $s\in (0,1)$. We also have that $C_c^\infty(\mathbb{R}^d)\subset D(\mathcal{L}_V)\subset D(\mathcal{L}_V^s)$, $s\in (0,1)$. Let $k\in \mathbb{N}$ and $f\in C_c^\infty(\mathbb{R}^d)$. We obtain
\begin{align*}
\int_0^\infty |\log\lambda|^{2k}d\mu_{f,f}(\lambda)&=\int_0^1\lambda|\log\lambda|^{2k}\lambda^{-1}d\mu_{f,f}(\lambda)+\int_1^\infty\lambda^{-1}|\log\lambda |^{2k}\lambda d\mu_{f,f}(\lambda),\\
&\le C\left(\|\mathcal{L}_V^{-1/2}f\|_{L^2(\mathbb{R}^d)}+\|\mathcal{L}_V^{1/2}f\|_{L^2(\mathbb{R}^d)}\right)<\infty.
\end{align*}
Then $f\in D(\log^k\mathcal{L}_V)$. By Proposition \ref{derlog} we conclude the desired result.
\end{proof}
\end{cor}

Note that by Corollary \ref{Ccreglog} we have that, for every $f\in C_c^\infty(\mathbb{R}^d)$ and $k\in \mathbb{N}$, 
\[
(\log^k\mathcal{L}_V)(f)(x)=(\mathbb{L}\textup{og}^k\mathcal L_V)(f)(x),\quad\mbox{a.e.}\,\,x\in \mathbb{R}^d.
\]

The next result justifies the definition of the operator $\mathbb{L}\textup{og}^k\mathcal L_V$, for every $k\in \mathbb N$, by means of the formula given in \eqref{Logk}.
\begin{prop}\label{propoY}\leavevmode
Let $k\in \mathbb N$ and $\theta \in (0,1)$. For every $f\in {\textup{Lip}}_\textup{loc}^\theta (\mathbb R^d)\cap L^1_0(\mathbb R^d)$, we have that
\begin{align}
    &\lim_{s\rightarrow 0^+}\partial _s^k\mathcal L_V^s(f)(x)\nonumber&\\
    &=\sum_{\ell =0}^k(-1)^\ell \binom{k}{\ell }\partial_s^{k-\ell }\left(\frac{1}{\Gamma (1-s)}\right)\Big|_{s=0}\nonumber\\
    &\quad \times\left[\ell \left(\int_0^\infty \frac{T_t^V(f)(x)-T_t^V(1)(x)f(x)}{t}(\log t)^{\ell -1}dt\right.\right.\nonumber\\
    &\quad +f(x)\int_0^{\rho (x)^2}\frac{T_t^V(1)(x)-T_t(1)(x)}{t}(\log t)^{\ell -1}dt\nonumber\\
    &\left. \left.\quad +f(x)\int_{\rho (x)^2}^\infty \frac{T_t^V(1)(x)}{t}(\log t)^{\ell -1}dt\right)+(2\log \rho (x))^\ell f(x)\right],\quad x\in \mathbb R^d.
\end{align}
\end{prop}
\begin{proof}
Let $f\in {\textup{Lip}}_\textup{loc}^\theta (\mathbb R^d)\cap L^1_0(\mathbb R^d)$. We have (see Remark \ref{rem}) 
\begin{align*}
\mathcal L_V^sf(x)&=-\frac{s}{\Gamma (1-s)}\int_0^\infty \frac{T_t^Vf(x)-f(x)}{t^{s+1}}dt=-\frac{s}{\Gamma (1-s)}\sum_{j=1}^4I_j(x,s)\\
&=-\frac{1}{\Gamma (1-s)}\left(s\sum_{j=1}^3I_j(x,s)-\frac{f(x)}{\rho (x)^{2s}}\right),\, x\in \mathbb R^d\mbox{ and }s\in \left(0,\min\left\{\tfrac{\theta}{2},\tfrac{\delta}{2}\right\}\right) ,
\end{align*}
where $I_j$, $j=1,2,3$, are given in \eqref{sumI}.

Then, 
\begin{align}\label{partial LV}
\partial _s^k\mathcal L_V^sf(x)&=-\sum_{\ell =0}^k\binom{k}{\ell }\partial_s^{k-\ell }\left(\frac{1}{\Gamma (1-s)}\right)\partial _s^\ell \left(s\sum_{j=1}^3I_j(x,s)-\frac{f(x)}{\rho (x)^{2s}}\right)\nonumber\\
&=-\sum_{\ell =0}^k\binom{k}{\ell }\partial_s^{k-\ell }\left(\frac{1}{\Gamma (1-s)}\right)\left(s\sum_{j=1}^3\partial _s^\ell I_j(x,s)+\ell \sum_{j=1}^3\partial _s^{\ell -1}I_j(x,s)\right.\nonumber\\
&\quad \left. -(-2\log \rho (x))^\ell \frac{f(x)}{\rho (x)^{2s}}\right),\quad x\in \mathbb R^d\mbox{ and }s\in \big(0,\min\left\{\tfrac{\theta}{2},\tfrac{\delta}{2}\right\}\big).
\end{align}
For $\ell =0$, $\partial _s^{\ell -1}$ is understood as the null operator. 

Let $r\in \mathbb N\cup\{0\}$. We have that, for every $x\in \mathbb R^d$, and $0<s<\min\{\frac{\theta}{2}-\varepsilon, \frac{\delta}{2}\}$, with $0<\varepsilon<\frac{\theta }{2}$, 
\begin{align}\label{sumJ}
\sum_{j=1}^3\partial _s^rI_j(x,s)&=(-1)^r\left(\int_0^\infty \int_{\mathbb R^d}(f(y)-f\mathcal X_{B(x,1)}(y))\frac{T_t^V(x,y)}{t^{s+1}}(\log t)^rdydt\right.\nonumber\\
&\quad +f(x)\int_0^{\rho (x)^2}\int_{\mathbb R^d}\frac{T_t^V(x,y)-T_t(x-y)}{t^{s+1}}(\log t)^rdydt\nonumber\\
&\quad \left.+f(x)\left(\int_{\rho (x)^2}^\infty\int_{B(x,1)}-\int_0^{\rho (x)^2}\int_{\mathbb R^d\setminus B(x,1)} \right)\frac{T_t^V(x,y)}{t^{s+1}}(\log t)^rdydt\right)\nonumber\\
&:=(-1)^r\sum_{j=1}^3J_j^{(r)}(x,s).
\end{align}

The differentiation under the integral sign can be justified by proceeding as in the proof of \eqref{1.7}, \eqref{1.8} and \eqref{1.9}.  Using \eqref{1.6} and  \eqref{acotTVimpr} (for $N=1$) and taking into account that $f\in {\textup{Lip}}^\theta _{\rm loc}(\mathbb R^d)$ we obtain
\begin{align*}
  |J_1^{(r)}(x,s)|&\leq C\left(\int_0^1\frac{|\log t|^r}{t^{d/2+s+1}}\left(\int_{|x-y|\leq 1}|f(y)-f(x)|e^{-\frac{|x-y|^2}{4t}}dy\right.\right.\\
  &\hspace{-1.5cm}\quad \left.+\int_{|x-y|>1}|f(y)|e^{-\frac{|x-y|^2}{4t}}dy\right)dt\\
 &\hspace{-1.5cm}\quad +\rho (x)\int_1^\infty \frac{|\log t|^r}{t^{d/2+s+3/2}}\left(\int_{|x-y|\leq 1}|f(y)-f(x)|e^{-\frac{|x-y|^2}{5t}}dy\right.\\
  &\hspace{-1.5cm}\quad \left.\left.+\int_{|x-y|>1}|f(y)|e^{-\frac{|x-y|^2}{5t}}dy\right)dt\right)\\
  &\hspace{-1.5cm}\leq C\left(\int_0^1\frac{|\log t|^r}{t^{1-\varepsilon}}dt\left(\int_{|x-y|\leq 1}|x-y|^{\theta -d-2s-2\varepsilon}dy+\int_{|x-y|>1}\frac{|f(y)|}{|x-y|^{d+2s+2\varepsilon}}dy\right)\right.\\
 &\hspace{-1.5cm}\left.\quad +\rho (x)\int_1^\infty \frac{(\log t)^r}{t^{3/2}}dt\left(\int_{|x-y|\leq 1}|x-y|^{\theta -d-2s}dy+\int_{|x-y|>1}\frac{|f(y)|}{|x-y|^{d+2s}}dy\right)\right)\\
  &\hspace{-1.5cm}\leq C\left(\frac{1}{\theta -2s-2\varepsilon}+\frac{\rho (x)}{\theta -2s}+(1+\rho (x))(1+|x|)^d\right)\\
  &\hspace{-1.5cm}\leq C(1+\rho (x))(1+|x|)^d\Big(\frac{1}{\theta -2s-2\varepsilon}+1\Big),\quad x\in \mathbb R^d,\;s\in \left(0,\tfrac{\theta}{2}-\varepsilon\right).
\end{align*}
On the other hand, using \eqref{TtVTt} we get that
\begin{align*}
    |J_2^{(r)}(x,s)|&\leq C\frac{|f(x)|}{\rho (x)^\delta}\int_0^{\rho (x)^2}\frac{|\log t|^r}{t^{s+1-\tfrac{\delta}{2}}}\int_{\mathbb R^d}\phi _t(x-y)dydt\\
    &\leq  C\frac{|f(x)|}{\rho (x)^\delta}\int_0^{\rho (x)^2}\frac{|\log t|^r}{t^{s+1-\tfrac{\delta}{2}}}dt<\infty,\quad x\in \mathbb R^d,\;s\in \left(0,\tfrac{\delta}{2}\right),
\end{align*}
and, by \eqref{acotTVimpr} (for $N=1$),
\begin{align*}
|J_3^{(r)}(x,s)|&\leq C\rho (x)|f(x)|\left(\!\int_{\rho (x)^2}^\infty\!\!\int_{|x-y|<1}+\! \int_0^{\rho (x)^2}\!\!\int_{|x-y|\geq 1}\right)\frac{e^{-\frac{|x-y|^2}{5t}}}{t^{d/2+s+3/2}}|\log t|^rdydt\\
&\leq C\rho (x)|f(x)|\left(\int_{\rho (x)^2}^\infty \frac{|\log t|^r}{t^{d/2+s+3/2}}dt\right.\\
&\left.\quad +\int_0^{\rho (x)^2}\frac{|\log t|^r}{\sqrt{t}}dt\int_{|x-y|\geq 1}\frac{dy}{|x-y|^{d+2s+2}}\right)<\infty ,\, x\in \mathbb R^d,\;s\in (0,1).
\end{align*}
Thus, for $0<s<\min\{\tfrac{\theta}{2}-\varepsilon,\tfrac{\delta}{2}\}$, \eqref{partial LV} can be rewritten as
\begin{align*}
\partial _s^k\mathcal L_V^sf(x)&=\sum_{\ell =0}^k\binom{k}{\ell }\partial_s^{k-\ell }\left(\frac{1}{\Gamma (1-s)}\right)\left((-1)^{\ell+1} s\sum_{j=1}^3J_j^{(\ell)}(x,s)\right.\nonumber\\
&\quad \left.+(-1)^{\ell}\ell \sum_{j=1}^3J_j^{(\ell -1)}(x,s)+(-2\log \rho (x))^\ell \frac{f(x)}{\rho (x)^{2s}}\right),\quad x\in \mathbb R^d.
\end{align*}
Here $J_{-1}$ is understood as $0$. Note that the above estimations are also valid in the case of $s=0$, so we can apply the dominated convergence theorem to conclude that
\begin{align}\label{derivLs}
\lim_{s\rightarrow 0^+}\partial _s^k\mathcal L_V^sf(x)&=\sum_{\ell =0}^k(-1)^\ell \binom{k}{\ell }\partial_s^{k-\ell }\left(\frac{1}{\Gamma (1-s)}\right)\Big|_{s=0}\nonumber\\
&\quad \times \left(\ell \sum_{j=1}^3J_j^{(\ell -1)}(x,0)+(2\log \rho(x))^\ell f(x)\right),\quad x\in \mathbb R^d,
\end{align}
where for $r\in \mathbb N$, and each $j=1,2, 3$, $J_j^{(r)}(x,0)$, $x\in \mathbb R^d$, denotes the integral given in \eqref{sumJ} for $s=0$. 
\end{proof}

\begin{prop}\label{Prop2.6}
Let $\theta \in (0,1)$, $\varepsilon \in (0,\frac{\theta}{2})$,  and $k\in \mathbb N$. For every $f\in {\textup{Lip}}^\theta_{\rm loc}(\mathbb R^d)\cap L^1_{-\varepsilon }(\mathbb R^d)$, 
\[
\lim_{s\rightarrow 0^+}\partial_s^k\mathcal L_V^{-s}f(x)=(-1)^k\lim_{s\rightarrow 0^+}\partial _s^k\mathcal L_V^sf(x),\quad x\in \mathbb R^d.
\]
\end{prop}
\begin{proof}
Let $f\in {\textup{Lip}}_{\rm loc}^\theta (\mathbb R^d)\cap L^1_{-\varepsilon}(\mathbb R^d)$. According to \eqref{sumR} we can write
\begin{align*}
   \partial _s^k\mathcal L_V^{-s}f(x)&=\sum_{j=1}^3\partial_s^k\left(\frac{sR_j(x,s)}{\Gamma (1+s)}\right)+\partial_s^k\left(\frac{f(x)\rho(x)^{2s}}{\Gamma (1+s)}\right)\\
   &=\sum_{\ell =0}^k\binom{k}{\ell }\partial _s^{k-\ell}\left(\frac{1}{\Gamma (1+s)}\right)\\
    &\quad \times \left(s\sum_{j=1}^3\partial _s^\ell R_j(x,s)+\ell \sum_{j=1}^3 \partial _s^{\ell -1}R_j(x,s)+\rho (x)^{2s}(2\log \rho (x))^\ell f(x)\right),
\end{align*}
for each $x\in \mathbb R^d$ and $s\in (0,\min\{\tfrac{\varepsilon}{2},1/4\})$.

Let $r\in \mathbb N\cup\{0\}$. We claim that, for each $x\in \mathbb R^d$ and $s\in (0,\min\{\frac{\varepsilon}{2}, \frac{1}{4}\})$,
\begin{align}\label{sumH}
    \sum_{j=1}^3\partial _s^rR_j(x,s)&=\left(\int_0^\infty \int_{\mathbb R^d}(f(y)-f\mathcal X_{B(x,1)}(y))T_t^V(x,y)t^{s-1}(\log t)^rdydt\right.\nonumber\\
&\quad +f(x)\int_0^{\rho (x)^2}\int_{\mathbb R^d}(T_t^V(x,y)-T_t(x-y))t^{s-1}(\log t)^rdydt\nonumber\\
&\quad \left.+f(x)\left(\int_{\rho (x)^2}^\infty\int_{B(x,1)}-\int_0^{\rho (x)^2}\int_{\mathbb R^d\setminus B(x,1)} \right)T_t^V(x,y)t^{s-1}(\log t)^rdydt\right)\nonumber\\
&:=\sum_{j=1}^3H_j^{(r)}(x,s).
\end{align}
Indeed, by proceeding as in the proof of Proposition \ref{propoY} we obtain
\begin{align*}
|H_1^{(r)}(x,s)|&\\
&\hspace{-1.5cm}\leq C\left(\int_0^1\frac{|\log t|^r}{t^{1-\varepsilon}}dt\left(\int_{|x-y|\leq 1}|x-y|^{\theta -d-2\varepsilon}dy+\int_{|x-y|>1}\frac{|f(y)|}{|x-y|^{d+\varepsilon}}dy\right)\right.\\
 &\hspace{-1.5cm}\left.\quad +\rho (x)\int_1^\infty \frac{(\log t)^r}{t^{5/4}}dt\left(\int_{|x-y|\leq 1}|x-y|^{\theta -d}dy+\int_{|x-y|>1}\frac{|f(y)|}{|x-y|^d}dy\right)\right)\\
 &\hspace{-1.5cm}\leq C(1+\rho (x))(1+|x|)^d,\quad x\in \mathbb R^d,\;s\in \left(0,\tfrac{1}{4}\right).
\end{align*}
We can also see that
\[
|H_2^{(r)}(x,s)|\leq C|f(x)|(1+\rho (x)^2)\rho (x)^{-\delta}\int_0^{\rho (x)^2}\frac{|\log t|}{t^{1-\delta/2}}dt<\infty,\quad x\in \mathbb R^d,\;s\in (0,1),
\]
and, for each $x\in \mathbb R^d$ and $s\in (0,\tfrac{1}{2})$,
\begin{align*}
|H_3^{(r)}(x,s)|&\leq C|f(x)|\rho (x)\left(\int_{\rho (x)^2}^\infty \frac{|\log t|^r}{t^{d/2-s+3/2}}dt\right.\\
&\left.\quad +\int_0^{\rho (x)^2}\frac{t^s|\log t|^r}{\sqrt{t}}dt\int_{|x-y|\geq 1}\frac{dy}{|x-y|^{d+2}}\right)\\
&\leq C|f(x)|\rho (x)\left((1+\rho (x)^2)\int_0^{\rho(x)^2}\frac{|\log t|^r}{\sqrt{t}}dt\right.\\
&\left.\quad+\left(\frac{1}{\rho (x)^d}+\frac{1}{\rho (x)^{d-1}}\right)\int_{\rho(x)^2}^\infty\frac{|\log t|^r}{t^{3/2}}dt\right)<\infty .
\end{align*}
Applying the dominated convergence theorem we conclude that
\begin{align*}
\lim_{s\rightarrow 0^+}\partial_s^k\mathcal L_V^{-s}f(x)&=\sum_{\ell =0}^k\binom{k}{\ell}\partial _s^{k-\ell}\left(\frac{1}{\Gamma (1+s)}\right)\Big|_{s=0}\\
&\quad \times \left(\sum_{j=1}^3\ell H_j^{(\ell -1)}(x,0)+f(x)(2\log \rho (x))^\ell \right),\quad x\in \mathbb R^d,
\end{align*}
where for $r\in \mathbb N\cup\{0\}$, $H_j^{(r)}(\cdot, 0)$, $j=1,2,3$, represents each term in \eqref{sumH} with $s=0$. Note that $H_j^{(r)}(\cdot, 0)=J_j^{(r)}(\cdot, 0)$, $j=1,2,3$, (see \eqref{sumJ}). Since, for each $r\in \mathbb N\cup\{0\}$, 
\[
\partial_s^r\left(\frac{1}{\Gamma (1+s)}\right)\Big|_{s=0}=(-1)^r\partial _s^r\left(\frac{1}{\Gamma (1-s)}\right)\Big|_{s=0},
\]
from \eqref{derivLs} we can finish the proof.
\end{proof}

\section{Proof of Theorem \ref{Th1.1}}
Let us show \ref{itm: Th1.1-a}. The property \ref{itm: Th1.1-b} can be proved in a similar way.

Assume that $f\in D(\log ^{m+1}\mathcal L_V)\cap D(\mathcal L_V^{s_0})$. Then, $f\in D(\mathcal L_V^s)$, for every $s\in (0,s_0)$. Let $k\in \mathbb N\cup\{0\}$, $k\leq m+1$. We have that
\[
\int_0^\infty |\log \lambda |^{2k}d\mu_{f,f}(\lambda )\leq \int_{e^{-1}}^ed\mu _{f,f}(\lambda) +\left(\int_0^{e^{-1}}+\int_e^\infty \right)|\log \lambda |^{2(m+1)}d\mu_{f,f}(\lambda )<\infty.
\]
Hence, $f\in D(\log ^k\mathcal L_V)$. We can write
\[
\mathcal L_V^sf-\sum_{k=0}^m\frac{s^k}{k!}(\log^k\mathcal L_V)f=\int_0^\infty \Big(\lambda ^s-\sum_{k=0}^m\frac{s^k}{k!}(\log \lambda )^k\Big)dE_V(\lambda)f.
\]
We define the function
\[
F(s,\lambda )=\lambda ^s-\sum_{k=0}^m\frac{s^k}{k!}(\log \lambda )^k,\quad \lambda \in (0,\infty )\mbox{ and }s\in (0,1).
\]
Using the mean value theorem we deduce that
\[
|F(s,\lambda )|\leq \frac{\lambda ^u|\log \lambda |^{m+1}}{(m+1)!}s^{m+1},\quad \lambda \in (0,\infty )\mbox{ and }s\in (0,1),
\]
for certain $u\in (0,s)$. Then,
\[
|F(s,\lambda )|\leq \frac{(1+\lambda ^{\frac{s_0}{2}})|\log \lambda |^{m+1}}{(m+1)!}s^{m+1},\quad \lambda \in (0,\infty )\mbox{ and }s\in \left(0,\tfrac{s_0}{2}\right).
\]
We have that
\begin{align*}
\int_0^\infty \big((1+\lambda ^{\frac{s_0}{2}})|\log \lambda |^{m+1}\big)^2d\mu_{f,f}(\lambda )&\leq C\left(\int_{e^{-1}}^ed\mu _{f,f}(\lambda) \right.\\
&\hspace{-4cm}\quad \left.+\int_0^{e^{-1}}|\log \lambda |^{2(m+1)}d\mu _{f,f}(\lambda)+\int_e^\infty\lambda ^{2s_0}d\mu_{f,f}(\lambda )\right)<\infty.
\end{align*}
It follows that
\[
\int_0^\infty |F(s,\lambda )|^2d\mu_{f,f}(\lambda )\leq Cs^{2(m+1)},\quad s\in \big (0,\tfrac{s_0}{2}\big).
\]
Thus, we establish that
\[
\lim_{s\rightarrow 0^+}\frac{\left\|\mathcal L_V^s f-\sum_{k=0}^m\frac{s^k}{k!}(\log ^k\mathcal L_V)f\right\|_{L^2(\mathbb R^d)}}{s^m}=0,
\]
and the proof is finished.

\section{Proof of Theorem \ref{Th1.2}}
We first present the following technical lemma, which will be useful in the proof of our result.
 \begin{lem}\label{techn}
Let $a>0$ and denote by $g_a$ and $h_a$ the functions given by
\[
g_a(s)=\frac{sa^{-s-1}}{\Gamma (1-s)},\quad \mbox{ and }\quad h_a(s)=\frac{a^{-2s}}{\Gamma (1-s)}-1, \quad s\in (-1,1).
\]
Then, for every $n\in \mathbb N$ and $\varepsilon >0$, there exists $C>0$ such that, for each $s\in (-1,1)$,
\[
\Big|g_a(s)-\sum_{k=1}^n\frac{d^kg_a}{ds^k}(0)\frac{s^k}{k!}\Big|\leq C\frac{|s|^{n+1}}{a^{\max\{s,0\}}}\Big(a^{-\varepsilon -1}\mathcal X_{(0,1]}(a)+a^{|s|+\varepsilon -1}\mathcal X_{(1,\infty)}(a)\Big),
\]
and 
\[
 \Big|h_a(s)-\sum_{k=1}^n\frac{d^kh_a}{ds^k}(0)\frac{s^k}{k!}\Big|\leq C\frac{|s|^{n+1}}{a^{\max\{2s,0\}}}\Big(a^{-\varepsilon}\mathcal X_{(0,1]}(a)+a^{|2s|+\varepsilon}\mathcal X_{(1,\infty)}(a)\Big).
\]
\end{lem}
\begin{proof}
Fix $n\in \mathbb N$. First, observe that
\begin{align}\label{derivg}
    \frac{d^{n+1}g_a}{du^{n+1}}(u)&=u\frac{d^{n+1}}{du^{n+1}}\left(\frac{a^{-u-1}}{\Gamma (1-u)}\right)+(n+1)\frac{d^n}{du^n}\left(\frac{a^{-u-1}}{\Gamma (1-u)}\right)\nonumber\\
    &=u\sum_{\ell =0}^{n+1}(-1)^\ell\binom{n+1}{\ell }\frac{d^{n+1-\ell}}{du^{n+1-\ell}}\left(\frac{a^{-u-1}}{\Gamma (1-u)}\right)\frac{(\log a)^\ell}{a^{u+1}}\\
    &\quad +(n+1)\sum_{\ell =0}^n(-1)^\ell\binom{n}{\ell}\frac{d^{n-\ell}}{du^{n-\ell}}\left(\frac{a^{-u-1}}{\Gamma (1-u)}\right)\frac{(\log a)^\ell}{a^{u+1}},\quad u\in (-1,1).\nonumber
\end{align}
Then, 
\[
\left|\frac{d^{n+1}g_a}{du^{n+1}}(u)\right|\leq C\frac{|u+1|}{a^{u+1}}\sum_{\ell =0}^{n+1}|\log a|^\ell\leq \frac{C}{a^{u+1}}(1+|\log a|^{n+1}),\quad u\in (-1,1).
\]

Let $\varepsilon >0$. Assume first that $s\in (0,1)$. The differential mean value theorem leads to
\begin{align*}
\left|g_a(s)-\sum_{k=1}^n\frac{d^kg_a}{ds^k}(0)\frac{s^k}{k!}\right|&\leq Cs^{n+1}\sup_{0\leq u\leq s}\left|\frac{d^{n+1}g_a(u)}{du^{n+1}}\right|\\
&\leq Cs^{n+1}(1+|\log a|^{n+1})\sup_{0\leq u\leq s}\frac{1}{a^{u+1}}\\
&\leq Cs^{n+1}\left(a^{-s-\varepsilon-1}\mathcal X_{(0,1]}(a)+a^{\varepsilon -1}\mathcal X_{(1,\infty )}(a)\right).
\end{align*}
In the case that $s\in (-1,0)$, we obtain that
\begin{align*}
\left|g_a(s)-\sum_{k=1}^n\frac{d^kg_a}{ds^k}(0)\frac{s^k}{k!}\right|&\leq Cs^{n+1}(1+|\log a|^{n+1})\sup_{s\leq u\leq 0}\frac{1}{a^{u+1}}\\
&\leq Cs^{n+1}(a^{-\varepsilon-1}\mathcal X_{(0,1]}(a)+a^{-s+\varepsilon -1}\mathcal X_{(1,\infty )}(a)).
\end{align*}
Thus, the estimation for $g_a$ is established. By proceeding in a similar way we get
\begin{align}\label{derivh}
\left|\frac{d^{n+1}h_a(u)}{du^{n+1}}\right|&=\left|\sum_{\ell =0}^{n+1}\binom{n+1}{\ell}\frac{d^{n+1-\ell}}{du^{n+1-\ell}}\left(\frac{a^{-u-1}}{\Gamma (1-u)}\right)a^{-2u}(-2\log a)^\ell\right|\\
&\leq \frac{C}{a^{2u}}(1+|\log a|^{n+1}),\quad u\in (-1,1),\nonumber
\end{align}
from which, as before, we deduce the property for $h_a$.
\end{proof}

We now proceed with the proof of Theorem \ref{Th1.2}. Let $n\in \mathbb N$. Suppose $f\in L^1(\mathbb R^d)\cap {\textup{Lip}}^\theta _{{\rm loc}, {\rm uni}}(\mathbb R^d)$. We define $\Omega_f(x,s)$, for each $x\in \mathbb R^d$ and $s\in (0,\min\left\{\tfrac{\theta}{2},\tfrac{\delta}{2}\right\})$, by
\[
\Omega _f(x,s):=\frac{1}{s^n}\left(\mathcal L_V^s(f)(x)-f(x)-\sum_{k=1}^n\frac{s^k}{k!}(\mathbb{L}\textup{og}^k\mathcal L_V)f(x)\right).
\]
Here $\delta=2-d/q$. Our objective is to show that $\lim_{s\rightarrow 0^+}\Omega_f(\cdot,s)=0$, in $L^p(\mathbb R^d)$, $1<p<\infty$, when $f\in C_c(\mathbb R^d)\cap {\textup{Lip}}_{\textup{loc},\textup{uni}}^\theta (\mathbb R^d)$, and in $L^\infty (\mathbb K)$, for every compact $\mathbb K\subset \mathbb R^d$, provided that $f\in L^1(\mathbb R^d)\cap L^\infty (\mathbb R^d)\cap {\textup{Lip}}_{\textup{loc},\textup{uni}}^\theta (\mathbb R^d)$.

From Remark~\ref{rem}, when $s\in (0,\min\left\{\tfrac{\theta}{2},\tfrac{\delta}{2}\right\})$, 
\[
\mathcal L_V^s(f)(x)-f(x)=-\frac{s}{\Gamma (1-s)}\sum_{j=1}^3I_j(x,s)+f(x)\left(\frac{\rho (x)^{-2s}}{\Gamma (1-s)}-1\right),\quad x\in \mathbb R^d,
\]
being $I_j$, $j=1,2,3$, the integrals given in \eqref{sumI}. On the other hand, according to \eqref{derivLs} if we set 
\begin{equation}\label{akl}
a_{k,\ell}=(-1)^\ell \binom{k}{\ell}\partial _s^{k-\ell}\left(\frac{1}{\Gamma (1-s)}\right)\Big|_{s=0},\quad k=1,\ldots, n,\;\ell =0,\ldots, k,
\end{equation}
we can write, for $x\in \mathbb R^d$,
\[(\mathbb{L}\textup{og}^k\mathcal L_V)f(x)=\sum_{\ell =0}^ka_{k,\ell}\left(\ell \sum_{j=1}^3J_j^{(\ell -1)}(x,0)+(2\log \rho (x))^\ell f(x)\right),\]
where, for every $r\in \mathbb N\cup\{0\}$, $J_j^{(r)}(\cdot,0)$, $j=1,2,3$, are the summands in \eqref{sumJ} when $s=0$ ($J^{-1}$ is understood as $0$). 
Then, for $x\in \mathbb R^d$ and $s\in (0,1)$,
\[
\sum_{k=1}^n\frac{s^k}{k!}(\mathbb{L}\textup{og}^k\mathcal L_V)f(x)=\sum_{k=1}^n\frac{s^k}{k!}\sum_{\ell =0}^ka_{k,\ell}\left(\ell \sum_{j=1}^3J_j^{(\ell -1)}(x,0)+(2\log \rho (x))^\ell f(x)\right).
\]
We decompose $\Omega_f(x,s)=\sum_{j=1}^4\Omega_{f,j}(x,s)$, $x\in \mathbb R^d$ and $s\in (0,\min\left\{\tfrac{\theta}{2},\tfrac{\delta}{2}\right\})$, where, for $j=1,2,3$, 
\[
\Omega_{f,j}(x,s)=\frac{1}{s^n}\left(-\frac{s}{\Gamma (1-s)}I_j(x,s)-\sum_{k=1}^n\frac{s^k}{k!}\sum_{\ell =1}^k\ell a_{k,\ell}J_j^{(\ell -1)}(x,0)\right),
\]
and
\[
\Omega_{f,4}(x,s)=\frac{f(x)}{s^n}\left(\frac{\rho (x)^{-2s}}{\Gamma (1-s)}-1-\sum_{k=1}^n\frac{s^k}{k!}\sum_{\ell =0}
^ka_{k,\ell}(2\log \rho (x))^\ell \right).
\]
Consider the functions $G$ and $H$ defined by
\begin{equation}\label{functionG}
G(s,t)=\frac{s}{\Gamma (1-s)t^{s+1}}+\sum_{k=1}^n\frac{s^k}{k!}\sum_{\ell =1}^k\ell a_{k,\ell}\frac{(\log t)^{\ell -1}}{t},\quad s\in (-1,1),\,t>0, 
\end{equation}
and
\begin{equation}\label{functionH}
H(s,x)=\frac{\rho (x)^{-2s}}{\Gamma (1-s)}-1-\sum_{k=1}^n\frac{s^k}{k!}\sum_{\ell =0}^ka_{k,\ell}(2\log \rho (x))^\ell,\quad s\in (-1,1),\,x\in \mathbb R^d.
\end{equation}
Observe that, if $g_t$, $t>0$, and $h_{\rho (x)}$, $x\in \mathbb R^d$, denote the functions defined in Lemma \ref{techn}, then, by virtue of \eqref{derivg} and \eqref{derivh} it follows that, for each $t>0$, $x\in \mathbb R^d$, and $k\in \mathbb N$,  
\[
\frac{d^kg_t}{ds^k}(0)=-\sum_{\ell =1}^k\ell a_{k,\ell }\frac{(\log t)^{\ell -1}}{t}, \mbox{ and }\quad
\frac{d^kh_{\rho (x)}}{ds^k}(0)=\sum_{\ell =0}^ka_{k,\ell }(2\log \rho (x))^\ell.
\]
Thus,
\[
G(s,t)=g_t(s)-\sum_{k=1}^n\frac{d^kg_t}{ds^k}(0)\frac{s^k}{k!},\quad s\in (-1,1),\,t>0,
\]
and
\[
H(s,x)=h_{\rho (x)}(s)-\sum_{k=1}^n\frac{d^kh_{\rho (x)}}{ds^k}(0)\frac{s^k}{k!},\quad s\in (-1,1),\,x\in \mathbb R^d.
\]
Note also that, for every $x\in \mathbb R^d$ and $s\in (0,\min\left\{\tfrac{\theta}{2},\tfrac{\delta}{2}\right\})$,
\begin{align*}
    &\Omega_{f,1}(x,s)=-\frac{1}{s^n}\int_0^\infty\int_{\mathbb R^d}\big(f(y)-f(x)\mathcal X_{B(x,1)}(y)\big)T_t^V(x,y)dyG(s,t)dt,\\
    &\Omega_{f,2}(x,s)=-\frac{f(x)}{s^n}\int_0^{\rho (x)^2}\int_{\mathbb R^d}(T_t^V(x,y)-T_t(x-y))dyG(s,t)dt,\\
    &\Omega_{f,3}(x,s)=-\frac{f(x)}{s^n}\left(\int_{\rho (x)^2}^\infty\int_{B(x,1)}-\int_0^{\rho (x)^2}\int_{\mathbb R^d\setminus B(x,1)} \right)T_t^V(x,y)dyG(s,t)dt,
\end{align*}
and
\[
\Omega_{f,4}(x,s)=\frac{f(x)}{s^n}H(s,x).
\]
Using Lemma \ref{techn} and \eqref{1.6} it follows that, for every $\varepsilon \in (0,\tfrac{\theta}{2})$, $s\in (0,\theta /4)$ and $x\in \mathbb R^d$, 
\begin{align}\label{Omega1a}
 |\Omega_{f,1}(x,s)|&\leq Cs\int_{\mathbb R^d} |f(y)-f(x)\mathcal X_{B(x,1)}(y)|\nonumber\\
 &\quad \times \left(\int_0^1\frac{e^{-\frac{|x-y|^2}{4t}}}{t^{d/2+s+1+\tfrac{\varepsilon}{2}}}dt+\int_1^\infty \frac{e^{-\frac{|x-y|^2}{t}}}{t^{d/2+1-\tfrac{\varepsilon}{2}}}dt\right)dy\nonumber\\
 &\leq Cs\left(\int_{B(x,1)}|f(y)-f(x)|\Big(\frac{1}{|x-y|^{d+2s+\varepsilon}}+1\Big)dy\right.\nonumber\\
 &\quad \left.+\int_{\mathbb R^d\setminus B(x,1)}|f(y)|\Big(\frac{1}{|x-y|^{d+2s+\varepsilon}}+\frac{1}{|x-y|^{d-\varepsilon}}\Big)dy\right)\nonumber\\
 &\leq Cs\left(\|f\|_{{\textup{Lip}}^\theta_{\rm loc,uni}(\mathbb R^d)}\int_{B(x,1)}|x-y|^{\theta-d-2s-\varepsilon}dy+\|f\|_{L^1(\mathbb R^d)}\right)\nonumber\\
 &=Cs\left(\|f\|_{{\textup{Lip}}^\theta_{\rm loc,uni}(\mathbb R^d)}\int_0^1r^{\theta-2s-\varepsilon-1}dy+\|f\|_{L^1(\mathbb R^d)}\right)\nonumber\\
 &\leq Cs\big(\|f\|_{{\textup{Lip}}^\theta_{\rm loc,uni}(\mathbb R^d)})+\|f\|_{L^1(\mathbb R^d)}\big).
\end{align}
Moreover, observe that, if $f\in C_c(\mathbb R^d)\cap {\textup{Lip}}^\theta_{\rm loc, uni}(\mathbb R^d)$, and $\supp f\subset B(0,R)$, with $R>1$, we can deduce that, for every $\varepsilon >0$, $|x|>2R$, and $s\in (0,1)$, 
\begin{align}\label{Omega1b}
 |\Omega_{f,1}(x,s)|&\leq Cs\int_{\mathbb B(0,R)} |f(y)|\left(\int_0^1\frac{e^{-\frac{|x-y|^2}{4t}}}{t^{d/2+s+1+\tfrac{\varepsilon}{2}}}dt+\int_1^\infty \frac{e^{-\frac{|x-y|^2}{4t}}}{t^{d/2+1-\tfrac{\varepsilon}{2}}}dt\right)dy\nonumber\\
 &\leq Cs\|f\|_{L^\infty (\mathbb R^d)}\int_{B(0,R)}\Big(\frac{1}{|x-y|^{d+2s+\varepsilon}}+\frac{1}{|x-y|^{d-\varepsilon}}\Big)dy\nonumber\\
  &\leq CR^ds\|f\|_{L^\infty (\mathbb R^d)}\Big(\frac{1}{|x|^{d+2s+\varepsilon}}+\frac{1}{|x|^{d-\varepsilon}}\Big)\leq C\frac{s\|f\|_{L^\infty (\mathbb R^d)}}{|x|^{d-\varepsilon}}.
\end{align}
According to \eqref{TtVTt} we can write, for each $x\in \mathbb R^d$ and $s\in (0,\min\left\{\tfrac{\theta}{2},\tfrac{\delta}{2}\right\})$,
\[
|\Omega_{f,2}(x,s)|\leq C\frac{|f(x)|}{s^n\rho (x)^\delta}\int_0^{\rho (x)^2}\int_{\mathbb R^d}t^{\tfrac{\delta}{2}}\left|\phi _t\left(\frac{x-y}{\sqrt{t}}\right)\right||G(s,t)|dydt.
\]
Assume that $x\in \mathbb R^d$ and $\rho (x)\leq 1$. By Lemma \ref{techn} we get, for every $\varepsilon \in (0,\delta)$, 
\begin{align}\label{Omega2a}
 |\Omega_{f,2}(x,s)|&\leq Cs\frac{|f(x)|}{\rho(x)^\delta}\int_0^{\rho(x)^2} \frac{dt}{t^{s+1+\tfrac{\varepsilon}{2}-\tfrac{\delta}{2}}}\leq Cs\frac{|f(x)|}{\rho (x)^\delta }\frac{\rho (x)^{\delta -2s-\varepsilon}}{\tfrac{\delta}{2}-s-\tfrac{\varepsilon}{2}}\nonumber\\
    &\leq Cs\frac{|f(x)|}{\rho (x)^\delta},\quad s\in \left(0,\min\left\{ \tfrac{\theta }{2},\tfrac{\delta-\varepsilon}{4}\right\}\right).
\end{align}
Suppose now that $x\in\mathbb R^d$ and $\rho (x)>1$. Using again Lemma \ref{techn} we get, for each $\varepsilon \in (0,\delta)$, 
\begin{align}\label{Omega2b}
 |\Omega_{f,2}(x,s)|&\leq Cs\frac{|f(x)|}{\rho(x)^\delta}\left(\int_0^1\frac{dt}{t^{s+1+\tfrac{\varepsilon}{2}-\tfrac{\delta}{2}}}+\int_1^{\rho (x)^2} \frac{dt}{t^{1-\tfrac{\varepsilon}{2}-\tfrac{\delta}{2}}}\right)\nonumber\\
    &\leq Cs\frac{|f(x)|}{\rho (x)^\delta}\left(\frac{1}{\delta -\varepsilon}+\rho (x)^{\varepsilon +\delta}\right)\nonumber\\
    &\leq C|f(x)|\rho (x)^\varepsilon,\quad s\in \left(0,\min\left\{ \tfrac{\theta }{2},\tfrac{\delta-\varepsilon}{4}\right\}\right).
\end{align}
By putting together \eqref{Omega2a} and \eqref{Omega2b} we get, for every $\varepsilon \in (0,\delta)$, $x\in \mathbb R^d$ and $s\in (0,\min\{\tfrac{\theta}{2},(\delta-\varepsilon)/4\})$,
\begin{equation}\label{Omega2}
|\Omega_{f,2}(x,s)|\leq Cs\frac{|f(x)|}{\rho (x)^\delta}\Big(\mathcal X_{(0,1]}(\rho (x))+\rho (x)^{\delta+\varepsilon}\mathcal X_{(1,\infty )}(\rho (x)) \Big).
\end{equation}
Let us now analyze the term $\Omega_{f,3}$. Suppose first that $x\in \mathbb R^d$ satisfies $\rho (x)\leq 1$.
From Lemma \ref{techn} and \eqref{1.6} we have that, for every $\varepsilon \in (0,1)$,
\begin{align*}
 &|\Omega_{f,3}(x,s)|\\
 &\leq \frac{|f(x)|}{s^n}\left(\int_{\mathbb R^d\setminus B(x,1)}\int_0^{\rho (x)^2} +\int_{B(x,1)}\left(\int_{\rho (x)^2}^1+\int_1^\infty\right)\right)T_t^V(x,y)|G(s,t)|dtdy\\
    &\leq Cs|f(x)|\left(\int_{\mathbb R^d\setminus B(x,1)}\int_0^{\rho (x)^2}\frac{e^{-\frac{|x-y|^2}{4t}}}{t^{d/2+s+1+\tfrac{\varepsilon}{2}}}dtdy\right.\\
    &\left.+\int_{B(x,1)}\left(\int_{\rho (x)^2}^1\frac{e^{-\frac{|x-y|^2}{4t}}}{t^{d/2+s+1+\tfrac{\varepsilon}{2}}}dt+\int_1^\infty \frac{dt}{t^{d/2+1-\tfrac{\varepsilon}{2}}}\right)dy\right)\\
    &\leq Cs|f(x)|\left(\int_{|x-y|\geq 1}\frac{dy}{|x-y|^{d+2s+2}}\int_0^1\frac{dt}{t^{\varepsilon /2}}+\int_{\rho (x)^2}^\infty\frac{dt}{t^{d/2+s+1+\varepsilon /2}}+1\right)\\
    &\leq Cs|f(x)|\Big(1+\rho (x)^{-(d
    +2s+\varepsilon)}\Big)\leq Cs|f(x)|\Big(1+\rho (x)^{-(d+2+\varepsilon)}\Big)\\
    &\leq Cs|f(x)|\rho (x)^{-(d+2+\varepsilon)}, 
\end{align*}
when $s\in (0,\min\left\{\tfrac{\theta}{2},\tfrac{\delta}{2}\right\})$. 

On the other hand, for $x\in \mathbb R^d$ such that $\rho (x)>1$, using again Lemma \ref{techn} and \eqref{1.6} we obtain for each $\varepsilon \in (0,1)$, 
\begin{align*}
 |\Omega_{f,3}(x,s)|&\\
 &\hspace{-1.5cm} \leq \frac{|f(x)|}{s^n}\left(\int_{\mathbb R^d\setminus B(x,1)}\Big(\int_0^1 +\int_1^{\rho (x)^2}\Big)+\int_{B(x,1)}\int_{\rho (x)^2}^\infty \right)T_t^V(x,y)|G(s,t)|dtdy\\
    &\hspace{-1.5cm}\leq Cs|f(x)|\left(\int_{\mathbb R^d\setminus B(x,1)}\Big(\int_0^1\frac{e^{-\frac{|x-y|^2}{4t}}}{t^{d/2+s+1+\tfrac{\varepsilon}{2}}}dt+\int_1^{\rho (x)^2}\frac{e^{-\frac{|x-y|^2}{4t}}}{t^{d/2+1-\tfrac{\varepsilon}{2}}}dt\Big)dy\right.\\
    &\hspace{-1.5cm}\left.\quad +\int_{B(x,1)}\int_{\rho (x)^2}^\infty \frac{e^{-\frac{|x-y|^2}{4t}}}{t^{d/2+1-\tfrac{\varepsilon}{2}}}dtdy\right)\\
     &\hspace{-1.5cm}\leq Cs|f(x)|\left(\int_{|x-y|\geq 1}\frac{dy}{|x-y|^{d+2s+2}}\int_0^1\frac{dt}{t^{\varepsilon /2}}+\int_{|x-y|\geq 1}\frac{dy}{|x-y|^{d+\varepsilon}}\int_1^{\rho (x)^2}\frac{dt}{t^{1-\varepsilon}}\right.\\
     &\hspace{-1.5cm}\left.\quad +\int_{\rho (x)^2}^\infty \frac{dt}{t^{d/2+1-\varepsilon /2}}\right)\leq Cs|f(x)|\left(1+\rho (x)^{2\varepsilon}+\rho (x)^{-d+\varepsilon}\right)\\
    &\hspace{-1.5cm}\leq Cs|f(x)|\rho (x)^{2\varepsilon}, \quad s\in \left(0,\min\left\{\tfrac{\theta }{2},\tfrac{\delta}{2}\right\}\right).
\end{align*}
Hence, we get that, for every $\varepsilon \in (0,1)$, 
\begin{equation}\label{Omega3}
|\Omega_{f,3}(x,s)|\leq Cs|f(x)|\Big(\rho (x)^{-(d+2+\varepsilon)}\mathcal X_{(0,1]}(\rho (x))+\rho (x)^{2\varepsilon}\mathcal X_{(1,\infty )}(\rho (x))\Big),
\end{equation}
for $x\in \mathbb R^d$ and $s\in (0,\min\left\{\tfrac{\theta}{2},\tfrac{\delta}{2}\right\})$.

Finally, taking into account the estimation for $h_{\rho(x)}$ in Lemma \ref{techn}, we obtain that, for each $\varepsilon >0$, 
\begin{equation}\label{Omega4}
|\Omega_{f,4}(x,s)|\leq Cs|f(x)|\left(\frac{\mathcal X_{(0,1]}(\rho (x))}{\rho (x)^{2+\varepsilon}}+\rho (x)^{2\varepsilon}\mathcal X_{(1,\infty )}(\rho (x)\right),
\end{equation}
for $x\in \mathbb R^d$ and $s\in (0,\min\left\{\tfrac{\theta}{2},\tfrac{\delta}{2}\right\})$.

Let $\mathbb K$ be a compact in $\mathbb R^d$. We can find $x_1,\ldots,x_\ell \in \mathbb R^d$ with $\ell \in \mathbb N$ such that
$\mathbb K\subset \cup_{i=1}^\ell B(x_i,\rho (x_i))$. By using \cite[Proposition 1]{DGMTZ} we deduce that, for a certain $C>1$,
\[
\frac{\rho (x_i)}{C}\leq \rho (y)\leq C\rho (x_i),\quad y\in B(x_i,\rho (x_i)),\,i=1,\ldots, \ell.
\]
Then, there exists $0<A<B$ such that $A\leq \rho (y)\leq B$, $y\in \mathbb K$. 

From \eqref{Omega1a}, \eqref{Omega2}, \eqref{Omega3} and \eqref{Omega4}, we deduce that, for every $f\in L^1(\mathbb R^d)\cap L^\infty(\mathbb R^d)\cap {\textup{Lip}}^\theta _{{\rm loc, uni}}(\mathbb R^d)$,
\begin{equation}\label{sumOmegaLinfty}
\left|\sum_{j=1}^4\Omega_{f,j}(x,s)\right|\leq Cs(1+|f(x)|),\quad x\in \mathbb R^d,\;s\in 
\left(0,\min\left\{\tfrac{\theta }{2},\tfrac{\delta}{4}\right\}\right).
\end{equation}
If $f\in C_c(\mathbb R^d)\cap {\textup{Lip}}^\theta _{{\rm loc, uni}}(\mathbb R^d)$, and $\supp f\subset B(0,R)$, with $R>1$, then by using \eqref{Omega1b}, \eqref{Omega2}, \eqref{Omega3} and \eqref{Omega4}, we get that, for each $\varepsilon >0$, 
\begin{equation}\label{sumOmegaLp}
\left|\sum_{j=1}^4\Omega_{f,j}(x,s)\right|=|\Omega_{f,1}(x,s)|\leq Cs|f(x)||x|^{-(d-\varepsilon)},
\end{equation}
provided that $|x|\geq 2R$, and $s\in (0,\min\{\tfrac{\theta}{2},\delta/4\})$. Here  $C>0$ does not depend on $x$ nor $s$.

Therefore, if $f\in L^1(\mathbb R^d)\cap L^\infty (\mathbb R^d)\cap {\textup{Lip}}^\theta _{{\rm loc, uni}}(\mathbb R^d)$ by virtue of \eqref{sumOmegaLinfty} we get $\|\Omega _V(\cdot ,s)\|_{L^\infty (\mathbb R^d)}\leq Cs$, and consequently,
\[
\lim_{s\rightarrow 0^+}\|\Omega _f(\cdot ,s)\|_{L^\infty (\mathbb R^d)}=0.
\]
Let us consider now $1<p<\infty$ and suppose that $f\in  C_c(\mathbb R^d)\cap {\textup{Lip}}^\theta _{{\rm loc, uni}}(\mathbb R^d)$. Choose $R>1$ such that ${\rm supp }\,f\subset B(0,R)$. Using \eqref{sumOmegaLinfty} and \eqref{sumOmegaLp} we deduce that, when $0<\varepsilon<d(1-1/p)$, 
\begin{align*}
\|\Omega_f(\cdot,s)\|_{L^p(\mathbb R^d)}&\leq \|\Omega_f(\cdot,s)\|_{L^p(B(0,2R))}+\|\Omega_f(\cdot,s)\|_{L^p(\mathbb R^d\setminus B(0,2R))}\\
&\leq C(1+\|f\|_{L^\infty (\mathbb R^d)})s\Big(1+\Big(\int_{|x|\geq 2R}\frac{dx}{|x|^{(d-\varepsilon)p}}\Big)^{1/p}\Big)\\
&\leq Cs\Big(1+\Big(\int_{2R}^\infty r^{d(1-p)+\varepsilon p-1}dr\Big)^{1/p}\Big)\leq Cs,\quad s\in \big(0,\min\{\tfrac{\theta}{2},\tfrac{\delta}{4}\}\big).
\end{align*}
Thus, it follows that
\[
\lim_{s\rightarrow 0^+}\|\Omega_f(\cdot ,s)\|_{L^p(\mathbb R^d)}=0.
\]

\section{Proof of Theorem \ref{Th1.3}}
We argue as in the proof of Theorem \ref{Th1.2}. Let $n\in \mathbb N$ and assume that $f\in L^1(\mathbb R^d)\cap {\textup{Lip}}^\theta _{{\rm loc}, {\rm uni}}(\mathbb R^d)$.

We define $\Lambda_f(x,s)$, for each $x\in \mathbb R^d$ and $s\in (0,1/4)$, as
\[
\Lambda _f(x,s):=\frac{1}{s^n}\Big(\mathcal L_V^{-s}(f)(x)-f(x)-\sum_{k=1}^n\frac{(-s)^k}{k!}(\mathbb{L}\textup{og}^k\mathcal L_V)f(x)\Big).
\]
Our aim is to establish that $\lim_{s\rightarrow 0^+}\Lambda_f(\cdot,s)=0$, in $L^p(\mathbb R^d)$, $1<p<\infty$, when $f\in C_c(\mathbb R^d)\cap {\textup{Lip}}_{\textup{loc},\textup{uni}}^\theta (\mathbb R^d)$, and in $L^\infty (\mathbb K)$, for every compact $\mathbb K\subset \mathbb R^d$, provided that $f\in L^1(\mathbb R^d)\cap L^\infty (\mathbb R^d)\cap {\textup{Lip}}_{\textup{loc},\textup{uni}}^\theta (\mathbb R^d)$. 

Let $\varepsilon \in (0,d)$. According to Remark \ref{rem} we have that, for every $s\in (0,\min\left\{\tfrac{\varepsilon }{2},\tfrac{1}{4}\right\})$,
\[
 \mathcal{L}_V^{-s}(f)(x)-f(x)=\frac{s}{\Gamma (1+s)}\sum_{j=1}^3R_j(x,s)+f(x)\Big(\frac{\rho (x)^{2s}}{\Gamma (1+s)}-1\Big),\quad x\in \mathbb R^d,
\]
where $R_j(x,s)$, $j=1,2,3$, are the terms in \eqref{sumR}. 

Consider $a_{k,\ell}$, $k=1,\ldots, n$, $\ell =0,\ldots, k$, as in \eqref{akl} and decompose $\Lambda _f$ as $\Lambda _f=\sum_{j=1}^4\Lambda _{f,j}$, where, for $j=1,2,3$, 
\[
\Lambda_{f,j}(x,s)=\frac{1}{s^n}\Big(\frac{s}{\Gamma (1+s)}R_j(x,s)-\sum_{k=1}^n\frac{(-s)^k}{k!}\sum_{\ell =1}^k\ell a_{k,\ell}J_j^{(\ell -1)}(x,0)\Big),
\]
when $x\in \mathbb R^d$ and $s\in (0,\min\left\{\tfrac{\varepsilon}{2},\tfrac{1}{4}\right\})$, and
\[
\Lambda_{f,4}(x,s)=\frac{f(x)}{s^n}\Big(\frac{\rho (x)^{2s}}{\Gamma (1+s)}-1-\sum_{k=1}^n\frac{(-s)^k}{k!}\sum_{\ell =0}^ka_{k,\ell}(2\log \rho (x))^\ell \Big),
\]
for $x\in \mathbb R^d$ and $s\in (0,1)$. Here $J_j^{(r)}(\cdot,0)$, $r\in \mathbb N\cap\{0\}$, is given by \eqref{sumJ} (for $s=0$). 

Let $G$ and $H$ the functions given by \eqref{functionG} and \eqref{functionH}, respectively. We have that, for every $x\in \mathbb R^d$ and $s\in (0,\min\left\{\tfrac{\varepsilon}{2},\tfrac{1}{4}\right\})$,
\begin{align*}
    &\Lambda_{f,1}(x,s)=-\frac{1}{s^n}\int_0^\infty\int_{\mathbb R^d}\big(f(y)-f(x)\mathcal X_{B(x,1)}(y)\big)T_t^V(x,y)dyG(-s,t)dt,\\
    &\Lambda_{f,2}(x,s)=-\frac{f(x)}{s^n}\int_0^{\rho (x)^2}\int_{\mathbb R^d}(T_t^V(x,y)-T_t(x-y))dyG(-s,t)dt,\\
    &\Lambda_{f,3}(x,s)=-\frac{f(x)}{s^n}\left(\int_{\rho (x)^2}^\infty\int_{B(x,1)}-\int_0^{\rho (x)^2}\int_{\mathbb R^d\setminus B(x,1)} \right)T_t^V(x,y)dyG(-s,t)dt,
\end{align*}
and
\[
\Lambda_{f,4}(x,s)=\frac{f(x)}{s^n}H(-s,x),\quad s\in (0,1).
\]
In order to estimate $\Lambda_{f,j}$, $j=1,2,3$, we use Lemma \ref{techn}. By \eqref{1.6} we obtain, when $0<\varepsilon<\theta$, for every $x\in \mathbb R^d$ and $s\in (0,(d-\varepsilon)/2)$,
\begin{align*}
|\Lambda_{f,1}(x,s)|&\leq Cs\int_{\mathbb R^d} |f(y)-f(x)\mathcal X_{B(x,1)}(y)|\\
 &\hspace{-1.2cm}\quad \times \left(\int_0^1\frac{e^{-\frac{|x-y|^2}{4t}}}{t^{d/2+1+\tfrac{\varepsilon}{2}}}dt+\int_1^\infty \frac{e^{-\frac{|x-y|^2}{5t}}}{t^{d/2-s+1-\tfrac{\varepsilon}{2}}}dt\right)dy\\
 &\hspace{-1.2cm}\leq Cs\int_{\mathbb R^d} |f(y)-f(x)\mathcal X_{B(x,1)}(y)|\Big(|x-y|^{-d-\varepsilon}+|x-y|^{-d+2s+\varepsilon}\Big)dy\\
&\hspace{-1.2cm}\leq Cs\left(\|f\|_{{\textup{Lip}}^\theta_{\rm loc,uni}(\mathbb R^d)}\int_{B(x,1)}|x-y|^{\theta-d-\varepsilon}dy+\int_{\mathbb R^d\setminus B(x,1)}\frac{|f(y)|}{|x-y|^{d-2s-\varepsilon}}dy\right)\\
&\hspace{-1.2cm}\leq Cs\left(\|f\|_{{\textup{Lip}}^\theta_{\rm loc,uni}(\mathbb R^d)}\int_0^1r^{\theta-\varepsilon-1}dy+\|f\|_{L^1(\mathbb R^d)}\right)\\
&\hspace{-1.2cm}\leq Cs\big(\|f\|_{{\textup{Lip}}^\theta_{\rm loc,uni}(\mathbb R^d)})+\|f\|_{L^1(\mathbb R^d)}\big).
\end{align*}

On the other hand, according to \eqref{TtVTt} we can write, for $0<\varepsilon <\delta=2-d/q$, and each $x\in \mathbb R^d$ and $s\in (0,1)$, 
\begin{align*}
|\Lambda_{f,2}(x,s)|&\leq C\frac{|f(x)|}{s^n\rho (x)^\delta}\int_0^{\rho (x)^2}t^{\tfrac{\delta}{2}}\int_{\mathbb R^d}\left|\phi _t\left(\frac{x-y}{\sqrt{t}}\right)\right||G(-s,t)|dydt\\
&\leq Cs\frac{|f(x)|}{\rho (x)^\delta}\left(\mathcal X_{(0,1]}(\rho (x))\int_0^{\rho (x)^2}t^{\tfrac{\delta}{2}-1-\tfrac{\varepsilon}{2}}dt\right.\\
&\quad +\left.\mathcal X_{(1,\infty )}(\rho (x))\Big(\int_0^1t^{\tfrac{\delta}{2}-1-\tfrac{\varepsilon}{2}}dt+\int_1^{\rho (x)^2}t^{\tfrac{\delta}{2}+s+\tfrac{\varepsilon}{2} -1}dt\Big)\right)\\
&\leq Cs|f(x)|\big(\mathcal X_{(0,1]}(\rho (x))\rho (x)^{-\varepsilon}+\mathcal X_{(1,\infty )}(\rho (x))\rho (x)^{2+\varepsilon}\big).
\end{align*}
For $\Lambda_{f,3}$ we have the following estimation,
\begin{align*}
 |\Lambda_{f,3}(x,s)|&\leq \frac{|f(x)|}{s^n}\\
 &\quad  \times \left(\int_{B(x,1)}\int_{\rho (x)^2}^\infty +\int_{\mathbb R^d\setminus B(x,1)}\int_0^{\rho (x)^2}\right)T_t^V(x,y)|G(-s,t)|dtdy.
 \end{align*}
 Assume that $x\in \mathbb R^d$ such that $\rho (x)\leq 1$. By \eqref{1.6}, for every $\varepsilon \in (0,d)$ and $s\in (0,(d-\varepsilon)/4)$,
 \begin{align*}
    |\Lambda_{f,3}(x,s)|&\leq Cs|f(x)|\left(\int_{B(x,1)}\Big(\int_{\rho (x)^2}^1\frac{e^{-\frac{|x-y|^2}{4t}}}{t^{d/2+1+\tfrac{\varepsilon}{2}}}dt+\int_1^\infty \frac{e^{-\frac{|x-y|^2}{4t}}}{t^{d/2-s+1-\varepsilon /2}}dt\Big)dy\right.\\
    &\left.\quad +\int_{\mathbb R^d\setminus B(x,1)}\int_0^{\rho (x)^2}\frac{e^{-\frac{|x-y|^2}{4t}}}{t^{d/2+1+\tfrac{\varepsilon}{2}}}dtdy\right)\\
    &\leq Cs|f(x)|\left(\rho (x)^{-(d+\varepsilon)}+\frac{1}{d-2s-\varepsilon}+\int_{|x-y|\geq 1}\frac{dy}{|x-y|^{d+\varepsilon}}\right)\\
    &\leq Cs|f(x)|\Big(1+\rho (x)^{-(d+\varepsilon)}\Big)\leq Cs|f(x)|\rho (x)^{-(d+\varepsilon)}. 
\end{align*}
In the case that $x\in \mathbb R^d$ satisfies $\rho (x)>1$, we get, for $s\in (0,\varepsilon /4)$, with $\varepsilon \in (0,d/2)$,
\begin{align*}
    |\Lambda_{f,3}(x,s)|&\leq Cs|f(x)|\left(\int_{B(x,1)}\Big(\int_{\rho (x)^2}^\infty \frac{e^{-\frac{|x-y|^2}{4t}}}{t^{d/2-s+1-\varepsilon /2}}dt\Big)dy\right.\\
    &\left.\quad +\int_{\mathbb R^d\setminus B(x,1)}\Big(\int_0^1\frac{e^{-\frac{|x-y|^2}{4t}}}{t^{d/2+1+\tfrac{\varepsilon}{2}}}dt+\int_1^{\rho (x)^2}\frac{e^{-\frac{|x-y|^2}{4t}}}{t^{d/2-s+1-\tfrac{\varepsilon}{2}}}dt\Big)dy\right)\\
    &\leq Cs|f(x)|\left(\rho (x)^{-(d-2s-\varepsilon)}+\rho(x)^{2\varepsilon}\int_{|x-y|\geq 1}\frac{dy}{|x-y|^{d-2s+\varepsilon}}\right)\\
    &\leq Cs|f(x)|\Big(\rho (x)^{-(d-2s-\varepsilon)}+\rho(x)^{2\varepsilon}\Big)\leq Cs|f(x)|\rho (x)^{3\varepsilon}. 
\end{align*}
Hence, for each $\varepsilon \in (0,d/2)$, 
\[
|\Lambda _{f,3}(x,s)|\leq Cs|f(x)|\big(\mathcal X_{(0,1]}(\rho (x))\rho (x)^{-(d+\varepsilon)}+\mathcal X_{(1,\infty )}(\rho (x))\rho (x)^{3\varepsilon}\big),
\]
provided that $x\in \mathbb R^d$ and $s\in (0,\varepsilon /4)$.

Finally, we observe that, for each $\varepsilon >0$,
\begin{align*}
|\Lambda _{f,4}(x,s)|&=\frac{|f(x)||H(-s,x)|}{s^n}\\
&\hspace{-1cm}\leq Cs|f(x)|\Big(\rho (x)^{-\varepsilon}\mathcal X_{(0,1]}(\rho (x))+\rho (x)^{2+\varepsilon}\mathcal X_{(1,\infty )}(\rho (x)\Big),\quad x\in \mathbb R^d,\,s\in (0,1).
\end{align*}

The above estimations allow us to establish the assertions in Theorem \ref{Th1.3} by proceeding as in the proof of Theorem \ref{Th1.2}.


\bibliographystyle{acm}

\def\cprime{$'$} \def\ocirc#1{\ifmmode\setbox0=\hbox{$#1$}\dimen0=\ht0 \advance\dimen0 by1pt\rlap{\hbox to\wd0{\hss\raise\dimen0 \hbox{\hskip.2em$\scriptscriptstyle\circ$}\hss}}#1\else {\accent"17 #1}\fi}

\end{document}